\documentclass[10pt]{article}%

\usepackage[utf8]{inputenc}

\usepackage{amsmath}
\usepackage{amsfonts}
 \usepackage{mathrsfs}
\usepackage{amssymb, color}
\usepackage[linkcolor=black,anchorcolor=black,citecolor=black]{hyperref}
\usepackage[numbers,sort&compress]{natbib}
\usepackage{graphicx}
\numberwithin{equation}{section}
\usepackage[body={15.5cm,21cm}, top=3cm]{geometry}%
\providecommand{\U}[1]{\protect\rule{.1in}{.1in}}
\providecommand{\U}[1]{\protect \rule{.1in}{.1in}}
\newtheorem{theorem}{Theorem}[section]

\newtheorem{definition}[theorem]{Definition}
\newtheorem{example}[theorem]{Example}

\newtheorem{lemma}[theorem]{Lemma}

\newtheorem{proposition}[theorem]{Proposition}
\newtheorem{remark}[theorem]{Remark}

\newtheorem{assumption}[theorem]{Assumption}
\newenvironment{proof}[1][Proof]{\noindent \textbf{#1.} }{\  \rule{0.5em}{0.5em}}

\def \E{\mathbf{E}}

\def \P{\mathbf{P}}

\usepackage [numbers,sort&compress] {natbib}

\begin{document}
	\title{Mean-Field Backward Stochastic Differential Equations with Nonlinear Resistance and Double Mean Reflections}
	\author{ 	Hanwu Li \thanks{Research Center for Mathematics and Interdisciplinary Sciences, Shandong University, Qingdao 266237, Shandong, China. lihanwu@sdu.edu.cn.}
	\thanks{Frontiers Science Center for Nonlinear Expectations (Ministry of Education), Shandong University, Qingdao 266237, Shandong, China.}
    \thanks{Shandong Province Key Laboratory of Financial Risk, Shandong University, Qingdao 266237, Shandong, China.}
    \and Jin Shi \thanks{Research Center for Mathematics and Interdisciplinary Sciences, Shandong University, Qingdao 266237, Shandong, China. }}
	\date{}
	\maketitle
	
	\begin{abstract}
      In this paper, we investigate mean-field backward stochastic differential equation ($\text{MFBSDE}$) with double mean reflections and nonlinear resistance. Specifically, the constraints are formulated in terms of the expectation of the solution, and a compensating term is incorporated into the generator.  We establish the existence and uniqueness for both the case of Lipschitz generator and the case where the generator is quadratic and the terminal value is bounded. Finally, when the compensating term is absolutely continuous, we study the well-posedness of a variant type of doubly mean reflected MFBSDE with nonlinear resistance, whose generator depends on the density function of the compensating term.
	\end{abstract}


    \textbf{MSC-classification}: 60H10
	
\section{Introduction}
Pardoux and Peng \cite{PP} first pioneered the study of backward stochastic differential equations (BSDEs), which take the following form:
\begin{equation}
 Y_t=\xi+\int_t^T f(s,Y_s,Z_s)ds-\int_t^T Z_s dB_s,\quad \forall t\in[0,T].   
\end{equation}
They established the existence and uniqueness of solutions for BSDEs under the assumptions that the generator $f$ is uniformly Lipschitz continuous and the terminal value $\xi$ is square integrable.  Over the past decades, BSDE theory has evolved into several major branches to accommodate diverse modeling requirements, including BSDEs with jumps, forward-backward SDEs (FBSDEs), and constrained BSDEs. Among these, to tackle optimal stopping problems in mathematical finance and provide probabilistic representations for partial differential equations (PDEs) with obstacles constraints,  El Karoui, Kapoudjian, Pardoux, Peng and Quenez \cite{EKPPQ} further proposed the notion of reflected backward stochastic differential equations (RBSDEs). 
 By incorporating a non-decreasing reflection process $K$ that forces the solution $Y$ to stay above a given stochastic barrier (obstacle) $L$, i.e. 
 \begin{align}\label{constraint pointwise}
   Y_t\geq L_t, \ \forall t\in[0,T],  
 \end{align}
 the primary aim is to identify the minimal solution $Y$, which is uniquely determined by the Skorokhod condition $\int_{0}^{T} (Y_t-L_t) dK_t=0$. In contrast to pointwise constraints in classical RBSDEs, Briand, Elie, and Hu  \cite{BEH} introduced $\text{BSDEs}$ with mean reflection, whose constraints are articulated in terms of the distribution of the solution. Specifically, given a loss function $l$, the mean reflection constraint is formulated as 
 \begin{align}\label{constraint expectation}
 \mathbf{E}[l(t,Y_t)]\geq 0, \ \forall t\in[0,T].
 \end{align}
 Unlike the classical reflected case, the force $K$ is required to be a deterministic function that adheres to the Skorokhod condition$\int_{0}^{T}\mathbf{E}[l(t,Y_t)]dK_t=0$. The authors proved the existence and uniqueness of the so-called deterministic flat solution for mean reflected $\text{BSDEs}$ and further applied the results to  super-hedging problems within the context of ongoing risk management. More recently, Li \cite{L} investigated a more general reflection constraint. For given two nonlinear loss functions $L,R$  with $L\leq R$, the first component $Y$  satisfies the following condition $$\mathbf{E}[L(t,Y_t)]\leq 0\leq \mathbf{E}[R(t,Y_t)], \ \forall t\in[0,T],$$ 
and the minimality conditions for the bounded variation function $K$ are defined accordingly.  For more relevant advances in mean reflected BSDEs, {we refer the readers to \cite{BH,FS,HHLLW,LW,NQW,QF}.} 

Motivated by the demand for modeling complex systems in which individual dynamics depend on the collective behavior of the population, 
the study of mean-field backward stochastic differential equations (MFBSDEs) was pioneered by Buckdahn, Djehiche, Li and Peng \cite{BDL} in their seminal work, which incorporates the distribution of solution to the generators $f$ for BSDEs. This framework was developed in the context of analyzing the limiting behavior of a high-dimensional system of forward-backward BSDEs. Subsequently, Buckdahn, Li and Peng \cite{BLP} further investigated $\text{MFBSDEs}$ within a more general framework. They established key results including the existence and uniqueness of solution, the comparison theorem and the stochastic representation for a nonlocal $\text{PDE}$. {Li \cite{Li'} firstly introduced reflected MFBSDEs, where the constraint is written as in \eqref{constraint pointwise}. Recently, Hu, Moreau and Wang \cite{HMW}  investigated the solvability of MFBSDEs with mean reflection, whose constraint is given in terms of the expectation of the solution as in \eqref{constraint expectation}. Correspondingly,  Li and Shi \cite{LS} extended the MFBSDEs to the case of double mean reflections.} 
For further insights into the reflected MFBSDEs, {we refer to works \cite{CHM,DD,DDZ,DEH,HMW',HMW} and the references therein.} 

{Inspired by Skorokhod equation-based methods for obstacle problems, and the practical demands for financial pricing and risk hedging under wealth constraints,}  Qian and Xu \cite{QX}  first extended the classical framework of RBSDEs by introducing a novel nonlinear resistance mechanism, where the reflecting force $K$ not only acts directly but also feeds back into the system's generator $f$  through a nonlinear mapping $H(K)$. This introduces a resistance effect, where the force $K$ influences the solution's dynamics in a more complex, nonlinear manner. 
{The authors  established a pathwise construction of solutions via Skorokhod’s equation, in contrast to conventional penalization methods. They also applied the developed theory to financial problems, including super-hedging under wealth constraints and optimal stopping with recursive utilities. The RBSDEs with resistance share some similarities with a variant of Skorokhod's obstacle problem introduced in \cite{BEK}, which was further studied by Ma and Wang \cite{MW}.} {Subsequently, Luo \cite{LP}  investigated the well-posedness of mean-reflected MFBSDEs with nonlinear resistance. The author addressed both Lipschitz and quadratic generator conditions and presented an application to superhedging problems subject to risk constraints.} 

The present paper is dedicated to the study of the following $\text{MFBSDE}$ with double mean reflections and nonlinear resistance over the time interval $[0,T]$ 
\begin{equation}\label{MFBSDEDMR}
\begin{cases}
Y_t=\xi+\int_t^T f(s,Y_s,\mathbf{P}_{Y_s},Z_s,\mathbf{P}_{Z_s},G_s(K))ds-\int_t^T Z_s dB_s+K_T-K_t, \\
\mathbf{E}[L(t,Y_t)]\leq 0\leq \mathbf{E}[R(t,Y_t)], \\
K_t=K^R_t-K^L_t,\ K^R,K^L\in I[0,T], \\
\int_0^T \mathbf{E}[R(t,Y_t)]dK_t^R=\int_0^T \mathbf{E}[L(t,Y_t)]dK^L_t=0,
\end{cases}
\end{equation}
where $G_s(\cdot)$ is a functional of $K$, which depends on the path of $K$ during time period $[0,s]$. To some extent, this feature makes the equation regarded as a forward-backward equation. This leads to the difficulty of constructing the solution by stitching the local ones on each small time obtained by the contraction mapping method. Motivated by \cite{LP}, we investigate the well-posedness of Eq. \eqref{MFBSDEDMR} using a so-called two-step method. More precisely, for both the Lipschitz case and the quadratic growth case, the key difference between the two cases lies in the time interval: in the Lipschitz case, the contraction mapping holds on a general fixed-length time interval, whereas in the quadratic growth case, the mapping is contractive only on a small time interval whose length depends on the bound of the component $Y$. Therefore, for the quadratic growth case, we first establish a uniform bound for the component $Y$ under an additional assumption on the generator $f$. Using this bound, we obtain the local existence and uniqueness of the solution for Eq. \eqref{MFBSDEDMR}. Furthermore, we have the global well-posedness of Eq. \eqref{MFBSDEDMR} under the given resistance term.
In contrast, for the Lipschitz case, the well-posedness of MFBSDE with double mean reflections (for the given resistance term) is already known, and we refer the reader to \cite{LS} for details. Finally, when the generator includes the nonlinear resistance term, we apply the contraction mapping principle to establish the well-posedness of the solution for Eq. \eqref{MFBSDEDMR} under both cases.

Lastly, we consider doubly mean reflected MFBSDEs with a density function, {where $G_{\cdot}(K)$ in Eq. \eqref{MFBSDEDMR} denotes the Radon-Nikodym derivative of 
$K$ with respect to the Lebesgue measure.  In fact, the targeted solution can be constructed based on the solution to doubly mean-reflected MFBSDEs without a resistance term. To this end, the first problem we aim to address is to identify the conditions ensuring the absolute continuity of $K$, where $(Y,Z,K)$ stands for the solution to the underlying doubly mean reflected MFBSDE. Roughly speaking, the desired result holds provided that the loss functions $L$ and $R$  are sufficiently smooth and $Z$ satisfies certain enhanced integrability conditions. Moreover, by virtue of Malliavin  calculus, we derive a sufficient condition that guarantees the integrability of $Z$. Further, under a specific structural assumption on the generator (see Eq. \eqref{k+k-}), we construct the solution to the density-dependent doubly mean reflected MFBSDE. In particular, if the loss functions are linear and the generator is separable with respect to the variable $k$ (see Eq. \eqref{repre of f}), the corresponding solution is  unique.
} 
 
This paper is organized as follows. In Section 2, we recall some basic results about the Skorokhod problem and a priori estimates for MFBSDE. In Section 3,  we study the existence and uniqueness of doubly mean reflected MFBSDEs with nonlinear resistance when the generators are Lipschitz continuous. Then, we consider the case of quadratic generators and bounded terminal values in Section 4. Finally, the well-posedness of the doubly mean reflected MFBSDEs with density functions is investigated in Section 5. 
 
\section{Preliminaries}

 Fix a finite time horizon $T>0$. Consider a filtered probability space $(\Omega,\mathcal{F},\mathbb{F},\P)$ satisfying the usual conditions of right continuity and completeness. Let $B$ be a $1$-dimensional standard Brownian motion. Actually, all the results in this paper still hold when we consider the multi-dimensional Brownian motion. For each $p>1$ and for any finite time interval $[u,v]\subset [0,T]$, we first introduce the following notations.
 
\begin{itemize}	
\item $\mathcal{L}^p_v$: the collection of real-valued $\mathcal{F}_{v}$-measurable random variable $\xi$ satisfying 
	$$\left \| \xi \right \|_{\mathcal{L}^p}=\mathbf{E} \left [ \left | \xi  \right | ^p \right ]^{\frac{1}{p} }< \infty. $$
	 
	\item $\mathcal{L}^{\infty}_v$: the collection of real-valued $\mathcal{F}_{v}$-measurable random variable $\xi$ satisfying
	 $$\left \| \xi  \right \|_{\mathcal{L}^\infty}=\operatorname{ess\,sup}\left |\mathbf{\xi} \left ( \omega  \right )   \right |<\infty. $$
	
	\item $\mathcal{H}^{p}\left[u,v\right]$: the collection of real-valued $\mathcal{F}$-progressively measurable processes $\left(z_{t}\right)_{u\leq t\leq v}$ satisfying 
	$$\left \| z \right \|_{\mathcal{H}^p }=\mathbf{E}\left [\left ( \int_{u}^{v}\left | z_{t} \right |^2 d t \right )^{\frac{p}{2}}   \right ]^{\frac{1}{p} }<\infty. $$
	
	\item $S^p\left[u,v\right]$: the collection of real-valued $\mathcal{F}$-adapted continuous processes $\left(y_{t}\right)_{u\leq t \leq v}$ satisfying 
	$$ \left \| y \right \|_{\mathcal{S}^p }=\mathbf{E}\left [ \sup _{t\in[u,v]} \left | y_{t} \right |^p \right ]^{\frac{1}{p}}  <\infty.$$
	
	\item $S^{\infty}\left[u,v\right]$: the collection of real-valued $\mathcal{F}$-adapted continuous processes $\left(y_{t}\right)_{u\leq t \leq v}$ satisfying
	$$\left \| y \right \|_{\mathcal{S}^{\infty } } =\operatorname{ess\,sup}_{(t,\omega)\in\left[u,v\right]\times\Omega}\left | y_t\left (\omega\right )  \right |  <\infty.$$
	
	\item $\mathcal{P}_{p}\left(\mathbb{R}\right)$: the collection of all probability measures over $\left(\mathbb{R},\mathcal{B}\left(\mathbb{R}\right)\right)$ with finite $p$-th moment, endowed with the $p$-Wasserstein distance $W_{p}$, where $p > 1$ is a positive integer and $$W_{p}(\mu,\nu)=\left(\inf_{X \sim \mu,Y \sim\nu}\mathbf{E}[\left|X-Y\right|]^p\right)^{\frac{1}{p}}.$$
	
	
	
	
	\item $\mathcal{T}_{t}[u,v]$: the collection of $\left[u,v\right]$-valued $\mathcal{F}$-stopping times $\tau$ such that $\tau \geq t$, $\mathbf{P}$-a.s. 
	
	\item $BMO\left[u,v\right]$: the collection of real-valued progressively measurable processes $\left(z_{t}\right)_{u\leq t \leq v}$ such that 
	$$\left \| z \right \|_{BMO}:=\sup_{\tau\in \mathcal{T}_{u}\left[u,v\right]}\operatorname{ess\,sup}_{\omega\in\Omega}\mathbf{E}_{\tau}\left[\int_{\tau}^{v}\left | z_{s} \right |^2 d s \right]^{\frac{1}{2}}<\infty. $$

	\item $C\left[u,v\right]$: the set of continuous functions from $\left[u,v\right]$ to $\mathbb{R}$.

	\item  $BV\left[u,v\right]$: the set of functions $K\in C\left[u,v\right]$ with $K_u=0$ and $K$ is of bounded variation on $\left[u,v\right]$.
	
	\item  $I\left[u,v\right]$: the set of functions in $C\left[u,v\right]$ starting from the origin which is nondecreasing.
 \end{itemize}
	
When the interval $[u,v]=[0,T]$, we always omit the time index. 
 For example, we write $\mathcal{H}^{p},\mathcal{S}^p,\mathcal{S}^{\infty}$ and so on. 

\subsection{Skorokhod problem}
In this section, we present the Skorokhod problem with nonlinear constraints, systematically addressing its definition, well-posedness, and essential properties of the resulting solutions. 
\begin{definition}\label{def1}
Let $s\in C[0,T]$,  and $l,r:[0,T] \times \mathbb{R}\rightarrow \mathbb{R}$ be two functions with $l\leq r$.  A pair of functions $(x,K)\in C[0,T]\times BV[0,T]$ is called a solution to the  Skorokhod problem for $s$ with nonlinear constraints $l,r$ ($(x,K)=\mathbb{SP}_l^r(s)$ for short) if 
\begin{itemize}
\item[(i)] $x_t=s_t+K_t$;
\item[(ii)] $l(t,x_t)\leq 0\leq r(t,x_t)$;
\item[(iii)]  $K_{0-}=0$ and $K$ has the decomposition $K=K^r-K^l$, where $K^r,K^l$ are nondecreasing functions satisfying  
\begin{align}
\int_0^{T} I_{\{l(s,x_s)<0\}}dK^l_s=0, \  \int_0^{T} I_{\{r(s,x_s)>0\}}dK^r_s=0.
\end{align}
\end{itemize}
\end{definition}

We propose the following assumption on the functions $l,r$.
\begin{assumption}\label{asslr}
The functions $l,r:[0,T] \times \mathbb{R}\rightarrow \mathbb{R}$ satisfy the following conditions
\begin{itemize}
\item[(i)] For each fixed $x\in\mathbb{R}$, $l(\cdot,x),r(\cdot,x)\in C[0,T]$;
\item[(ii)] For any fixed $t\geq 0$, $l(t,\cdot)$, $r(t,\cdot)$ are strictly increasing;
\item[(iii)] There exist two positive constants $0<c<C<\infty$, such that for any $t\in [0,T]$ and $x,y\in \mathbb{R}$,
\begin{align*}
&c|x-y|\leq |l(t,x)-l(t,y)|\leq C|x-y|,\\
&c|x-y|\leq |r(t,x)-r(t,y)|\leq C|x-y|.
\end{align*} 
\item[(iv)] $\inf_{(t,x)\in[0,T]\times\mathbb{R}}(r(t,x)-l(t,x))>0$.
\end{itemize}
\end{assumption}

\begin{theorem}[\cite{Li}]\label{SP}
Suppose that $l,r$ satisfy Assumption \ref{asslr}. For any given $s\in C[0,T]$, there exists a unique pair of solution to the Skorokhod problem $(x,K)=\mathbb{SP}_l^r(s)$. Moreover, for any $t\geq 0$, let $\phi_t,\psi_t$ be the unique solutions to the following equations, respectively, 
    \begin{align*}
        l(t,s_t+x)=0,\ r(t,s_t+x)=0.
    \end{align*}
    Then, we have 
    \begin{align*}
        &K_t=-\max\left((-\phi_0)^{+}\wedge  \inf_{r\in[0,T]}(-\psi_r), \sup_{s\in[0,T]}\left[(-\phi_s) \wedge  \inf_{r\in[s,t]}(-\psi_r)\right]\right),\\
    \end{align*}
    and 
   \begin{align*}
        &\sup_{t\in[0,T]}\left|K_t\right|\leq \sup_{t\in[0,T]}\left|\phi_t\right|+\sup_{t\in[0,T]}\left|\psi_t\right|.\\
    \end{align*} 
\end{theorem}

The following proposition provides the continuity property of the solution  to the Skorokhod problem with respect to input function $s$ and the reflecting boundary functions $l,r$.

\begin{proposition}[\cite{Li}]\label{continuity}
Suppose that $(l^i,r^i)$ satisfy Assumption $\ref{asslr}$, $i=1,2$.
Given  $s^i\in C[0,T]$, let $(x^i,K^i)$ be the solution to the Skorokhod problem $\mathbb{SP}_{l^i}^{r^i}(s^i)$. Then, we have
\begin{equation}\label{diffk}
\sup_{t\in[0,T]}\left|K^1_t-K^2_t\right|
\leq \frac{C}{c}\sup_{t\in[0,T]}\left|s^1_t-s^2_t\right|+\frac{1}{c}(\bar{L}_T\vee\bar{R}_T),
\end{equation}
where
\begin{align*}
\bar{L}_T=\sup_{(t,x)\in[0,T]\times \mathbb{R}}\left|{l}^1(t,x)-{l}^2(t,x)\right|,\
\bar{R}_T=\sup_{(t,x)\in[0,T]\times \mathbb{R}}\left|{r}^1(t,x)-{r}^2(t,x)\right|.
\end{align*}
\end{proposition}


\subsection{Backward Skorokhod problem}
In this section, we introduce the backward Skorokhod problem with nonlinear constraints, including its definition, well-posedness, and key properties of its solution, while also clarifying its relationship with the classical Skorokhod problem.

\begin{definition}\label{def2}
Let $s\in C[0,T]$, $a\in \mathbb{R}$ and $l,r:[0,T]\times \mathbb{R}\rightarrow \mathbb{R}$ be two functions such that $l\leq r$ and $l(T,a)\leq 0\leq r(T,a)$. A pair of functions $(x,K)\in C[0,T]\times BV[0,T]$ is called a solution of the backward Skorokhod problem for $s$ with nonlinear constraints $l,r$ ($(x,K)=\mathbb{BSP}_l^r(s,a)$ for short) if 
\begin{itemize}
\item[(i)] $x_t=a+s_T-s_t+K_T-K_t$;
\item[(ii)] $l(t,x_t)\leq 0\leq r(t,x_t)$, $t\in[0,T]$;
\item[(iii)]  $K$ has the decomposition $K=K^r-K^l$, where $K^r,K^l\in I[0,T]$ satisfy 
\begin{align}\label{iii}
\int_0^T I_{\{l(s,x_s)<0\}}dK^l_s=0, \  \int_0^T I_{\{r(s,x_s)>0\}}dK^r_s=0.
\end{align}
\end{itemize}
\end{definition}


\begin{theorem}[\cite{L}]\label{BSP}
Let Assumption \ref{asslr} holds. For any given $s\in C[0,T]$ and $a\in \mathbb{R}$ with $l(T,a)\leq 0\leq r(T,a)$, there exists a unique solution to the backward Skorokhod problem $(x,k)=\mathbb{BSP}_l^r(s,a)$. 
\end{theorem}
\begin{remark}\label{rem2.7}
Set
\begin{equation}
    \begin{split}
     &\bar{s}_t=a+s_T-s_{T-t},\ x_t=\bar{x}_{T-t},\ K_t=\bar{K}_{T}-\bar{K}_{T-t}, \ t\in[0,T],\\ 
     &\bar{l}\left(t,x\right)=l\left(T-t,x\right),\ \bar{r}\left(t,x\right)=r\left(T-t,x\right), \ (t,x)\in[0,T]\times\mathbb{R}.
    \end{split}
\end{equation}
According to Theorem 3.9 in \cite{L},  $\left(\bar{x},\bar{K}\right)$ is the unique solution to the Skorokhod problem $\mathbb{SP}^{\bar{r}}_{\bar{l}}\left(\bar{s}\right)$.
\end{remark}   
 
The following proposition provides the continuous dependence of the solution with respect to the input function $s$ and the reflecting boundary functions $l,r$.

\begin{proposition}[\cite{L}]\label{continuous}
Given $a^i\in\mathbb{R}$, $s^i\in C[0,T]$, $l^i,r^i$ satisfy Assumption \ref{asslr} and $l^i(T,a^i)\leq 0\leq r^i(T,a^i)$, $i=1,2$, 
 let $(x^i,k^i)$ be the solution to the backward Skorokhod problem $\mathbb{BSP}_{l^i}^{r^i}(s^i,a^i)$, $i=1,2$. Then, we have
\begin{equation}\label{diffK}
\sup_{t\in[0,T]}\left|K^1_t-K^2_t\right|
\leq  2\frac{C}{c}\left|a^1-a^2\right|+4\frac{C}{c}\sup_{t\in[0,T]}\left|s^1_t-s^2_t\right|+\frac{2}{c}(\bar{L}_T\vee\bar{R}_T),
\end{equation}
where $\bar{L}_T$ and $\bar{R}_T$ are the same as in Proposition \ref{continuity}.
\end{proposition}
\subsection{A priori estimates for MFBSDEs}

In this subsection, we consider the MFBSDEs taking the following form:
\begin{equation}\label{MFBSDE}
    Y_t=\xi+\int_{t}^{T}f(s,Y_s, \mathbf{P}_{Y_s},Z_s,\mathbf{P}_{Z_s})ds-\int_{t}^{T}Z_sdB_s.
\end{equation}
We propose the following assumptions for the generator $f$ and the terminal value $\xi$.

\begin{assumption}\label{asslip2}
  \begin{itemize}
   \item [(i)] The terminal value $\xi\in \mathcal{L}^p_T$,$ \ p>1$;
      \item [(ii)] The process $\left\{f_t(0):=f(t,0,\delta_0,0,\delta_0)\right\}_{t\in[0,T]}\in \mathcal{H}^p$. There exists a constant $\lambda>0$ such that for all $t\in[0,T]$, $y, y'\in \mathbb{R}$, $\mu, \mu' \in \mathcal{P}_1(\mathbb{R})$, $z, z'\in \mathbb{R}$, $\nu, \nu'\in\mathcal{P}_1(\mathbb{R})$,
      \begin{align*}
\left|f(t,y,\mu,z,\nu)-f(t,y',\mu',z',\nu')\right|\leq \lambda\left(\left|y-y'\right|+\left|z-z'\right|+W_1(\mu,\mu')+W_1(\nu,\nu')\right).
\end{align*}
  \end{itemize}  
\end{assumption}

Similar to Proposition 3.3 in \cite{LX}, we have the following a priori estimates for MFBSDEs. 
\begin{lemma}\label{proMFBSDE}
Let Assumption \ref{asslip2} hold. Assume $(Y, Z)$ is the solution of the mean-field BSDE \eqref{MFBSDE}. Then, there exists a constant $M$ depending on the constant $\lambda,\ p,\ T$ such that
\begin{align}
&\E\left[ \sup_{s \in [0,T]} |Y_s|^p \right] + \E\left[\left(\int_0^T |Z_s|^2ds\right)^{\frac{p}{2}} \right]\leq M(\lambda, p, T) \E\left[ |\xi|^p + \left( \int_0^T |f_s(0)| ds \right)^p\right].
\end{align}
\end{lemma}

\section{Doubly mean reflected MFBSDEs with nonlinear resistance: the Lipschitz case}

The main purpose of this paper is to study the MFBSDE with nonlinear resistance and double mean reflections of the following type
\begin{equation}\label{nonlinearyz}
\begin{cases}
Y_t=\xi+\int_t^T f(s,Y_s,\P_{Y_s},Z_s,\P_{Z_s},G_s(K))ds-\int_t^T Z_s dB_s+K_T-K_t, \\
\E[L(t,Y_t)]\leq 0\leq \E[R(t,Y_t)], \\
K_t=K^R_t-K^L_t,\ K^R,K^L\in I[0,T],\\
\int_0^T \E[R(t,Y_t)]dK_t^R=\int_0^T \E[L(t,Y_t)]dK^L_t=0.
\end{cases}
\end{equation}
The loss functions $L,R:\Omega\times [0,T]\times\mathbb{R}\rightarrow \mathbb{R}$ are measurable maps with respect to $\mathcal{F}_T\times \mathcal{B}([0,T])\times \mathcal{B}(\mathbb{R})$ satisfying the following conditions.

\begin{assumption}\label{assLR}
\begin{itemize}
\item[(1)] For any fixed $(\omega,x)\in \Omega\times\mathbb{R}$, $L(\omega, \cdot, x), R(\omega, \cdot, x)$ are continuous;
\item[(2)]  There exists a constant $M>0$ such that $$\mathbf{E}\left[\sup_{t\in[0,T]}\left|L(t,0)\right|+\sup_{t\in[0,T]}\left|R(t,0)\right|\right]\leq M.$$
\item[(3)] For any fixed $(\omega,t)\in \Omega\times [0,T]$, $L(\omega,t,\cdot),R(\omega,t,\cdot)$ are strictly increasing and there  exist two constants $c,C$ satisfying $0<c<C$ such that for any $x,y\in \mathbb{R}$,
\begin{align*}
&c\left|x-y\right|\leq \left|L(\omega,t,x)-L(\omega,t,y)\right|\leq C\left|x-y\right|,\\
&c\left|x-y\right|\leq \left|R(\omega,t,x)-R(\omega,t,y)\right|\leq C\left|x-y\right|;
\end{align*}
\item[(4)] $\inf_{\omega,t,x} \left(R(\omega,t,x)-L(\omega,t,x)\right)>0$.
\end{itemize}
\end{assumption}
Compared with the doubly mean reflected BSDEs studied in \cite{L}, the generator $f$ in \eqref{nonlinearyz} depends not only on the distribution of the solution, but also on the compensating term $K$, which is required to satisfy the following assumption.
\begin{assumption}\label{ass1ipterminalp}
\begin{itemize}
      \item[(i)]  The terminal value $\xi \in \mathcal{L}^{p}_T $ satisfies $\mathbf{E}[L(T,\xi)]\leq 0 \leq \mathbf{E}[R(T,\xi)]$; 
      \item[(ii)] The generator $f$: $[0,T]\times\Omega\times \mathbb{R}\times \mathcal{P}_1(\mathbb{R})\times\mathbb{R}\times \mathcal{P}_1(\mathbb{R})\times \mathbb{R}$ $\rightarrow$ $\mathbb{R}$ is a $\mathcal{P}\times\mathcal{B}(\mathcal{P}_1(\mathbb{R}))\times\mathcal{B}(\mathbb{R})\times\mathcal{B}(\mathcal{P}_1(\mathbb{R}))\times\mathcal{B}(\mathbb{R})$-measurable map, and there exists a constant $\lambda>0$ such that for all $t\in[0,T]$, $y, y', k, k' \in \mathbb{R}$, $\mu, \mu' \in \mathcal{P}_1(\mathbb{R})$, $z, z'\in \mathbb{R}$, $\nu, \nu'\in\mathcal{P}_1(\mathbb{R})$, 
\begin{align*}
\left|f(t,y,\mu,z,\nu,k)-f(t,y',\mu',z',\nu',k')\right|\leq \lambda\left(\left|y-y'\right|+\left|z-z'\right|+W_1(\mu,\mu')+W_1(\nu,\nu')+\left|k-k'\right|\right).
\end{align*}
   \end{itemize}
\end{assumption}

For each $t\in[0,T]$, the resistance function $G_t:  BV[0,T]\rightarrow \mathbb{R}$ satisfies the following conditions.
\begin{assumption}\label{assgk}
For any $t\in[0,T]$, given $y\in BV[0,T]$,  we have $G_{t}(y)=G_{t}\left ( \left \{ y_{s\wedge t} \right \}_{0\le s\le t\le T}  \right ) $ and  $G_t(0)=0$. Moreover, there exists a constant $\lambda>0$, such that for any $y,y'\in BV[0,T]$,   $$\left | G_{t}(y)-G_t(y') \right | \leq \lambda \sup _{u\in[0,t]}\left | y_u-y'_{u} \right |.$$
\end{assumption}

\begin{example}
  We give some simple  examples of  $G_t(\cdot)$ satisfying Assumption \ref{assgk}.

  \begin{itemize}
      \item[(i)] $G_t(y)=y_t$.
      \item[(ii)] $G_t(y)=\sup_{0\leq s\leq t}y_s$.
      \item[(iii)] $G_{t}(y)=\int_{0}^{t}y_sds$.
  \end{itemize}
\end{example}

The construction of the solution to doubly mean reflected MFBSDE with nonlinear resistance \eqref{nonlinearyz} relies on the one of the solution to doubly mean reflected BSDE with constant generator as listed below. 
\begin{proposition}\label{pro-1}
  Let  Assumption \ref{ass1ipterminalp} and \ref{assLR} hold. Given $C\in \mathcal{H}^p$, the BSDE with double mean reflections 
   \begin{equation}\label{MFBSDEC}
   \begin{cases}
Y_t=\xi+\int_t^T C_s ds-\int_t^T Z_s dB_s+K_T-K_t, \\
\E[L(t,Y_t)]\leq 0\leq \E[R(t,Y_t)], \\
K_t=K^{R}_t-K^{L}_t, \int_0^T \E[R(t,Y_t)]dK_t^{R}=\int_0^T \E[L(t,Y_t)]dK^{L}_t=0,
\end{cases}
\end{equation}
has a unique solution $\left(Y,Z,K\right)\in \mathcal{S}^p\times\mathcal{H}^{p}\times BV[0,T]$.
\end{proposition}
\begin{proof}
 The proof is similar to the one for Proposition 3.4 in \cite{LS}. For the purpose of the remaining proof, we briefly describe the construction here.
 
Let $(\tilde{Y},Z)\in \mathcal{S}^p\times \mathcal{H}^p$ be the solution to the BSDE with terminal value $\xi$ and constant generator $C$. For any $t\in[0,T]$, set 
 $$\begin{aligned}
  & s_t=\E\left[\int_{0}^{t}C_s ds\right], a=\E\left[\xi\right]
 \end{aligned}$$
and for any $(t,x)\in[0,T]\times \mathbb{R}$, we define 
$$l(t,x):=\E\left[L\left(t,\tilde{Y}_t-\E[\tilde{Y}_t]+x\right)\right],r(t,x):=\E\left[R\left(t,\tilde{Y}_t-\E[\tilde{Y}_t]+x\right)\right].$$
It is easy to check that $s$ is a continuous function  and $l,r$ satisfy Assumption \ref{asslr}. By Theorem \ref{BSP}, the backward Skorokhod problem $\mathbb{BSP}_l^r(s,a)$ admits a unique solution $(x,K)$. Set$$Y_t=\tilde{Y}_t+K_T-K_t=\xi+\int_{t}^{T}C_s ds-\int_{t}^{T}Z_s d B_s+K_T-K_t.$$ 
Then, $(Y,Z,K)$ is the solution to \eqref{MFBSDEC}.
\end{proof}

Now, we are ready to present the main result in this section.
\begin{theorem}\label{thlip}
 Let Assumptions $\ref{ass1ipterminalp}-\ref{assgk}$ hold. Then,  the doubly mean reflected MFBSDE with nonlinear resistance \eqref{nonlinearyz} admits a unique solution $(Y,Z,K)\in \mathcal{S}^p \times \mathcal{H}^{p} \times BV[0,T].$  
\end{theorem}

We apply a contraction mapping argument to prove Theorem \ref{thlip}. For this purpose, for each fixed $(y,z,k)\in \mathcal{S}^p \times \mathcal{H}^{p} \times BV[0,T]$, according to Proposition \ref{pro-1},   the following BSDE with double mean reflections 
\begin{equation}\label{yzk}
\begin{cases}
Y_t^{y,z,k}=\xi+\int_t^T f(s,y_s,\P_{y_s},z_s,\P_{z_s}, G_s(k))ds-\int_t^T Z_s^{y,z,k} dB_s+K_T^{y,z,k}-K_t^{y,z,k}, \\
\E[L(t,Y_t^{y,z,k})]\leq 0\leq \E[R(t,Y_t^{y,z,k})], \\
K^{y,z,k}_t=K^{y,z,k,R}_t-K^{y,z,k,L}_t, \\
\int_0^T \E[R(t,Y_t^{y,z,k})]dK_t^{y,z,k,R}=\int_0^T \E[L(t,Y_t^{y,z,k})]dK^{y,z,k,L}_t=0
\end{cases}
\end{equation}
has a unique solution $(Y^{y,z,k},Z^{y,z,k},K^{y,z,k})\in \mathcal{S}^p \times \mathcal{H}^{p} \times BV[0,T]$. Then, we define the  map
\begin{align*}
\Gamma:\mathcal{S}^p\times\mathcal{H}^{p}\times BV[0,T]&\rightarrow \mathcal{S}^p\times\mathcal{H}^{p}\times BV[0,T],  \\
\Gamma(y^i,z^i,k^i)&=(Y^i,Z^i,K^i).
\end{align*}
We first show that $\Gamma$ is a contraction mapping  when $T$ is sufficiently small. 
\begin{lemma}\label{le3-2}
Let $\varepsilon$ satisfy $\beta(\lambda,c,C,p)\max\left(\varepsilon^{\frac{p}{2}},\varepsilon^p\right)<1$, where $\beta$ is a positive constant depending only on $\lambda, C,c,p$. For $T\in (0,\varepsilon]$, the solution map $\Gamma$ is contractive. 
\end{lemma}

\begin{proof}
In this proof,  $\beta$ always represents a positive constant depending on $\lambda,C,c,p$, which may vary from line to line. Let $(Y^i,Z^i,K^i):=(Y^{y^i,z^i,k^i},Z^{y^i,z^i,k^i},K^{y^i,z^i,k^i})$ be the solution to BSDE  with double mean reflections \eqref{yzk} associated with data $(y^i,z^i,k^i)$, $i=1,2$.

Set
\begin{align*}
&\hat{F}_t=F^1_t-F^2_t, \textrm{ where } F=Y,Z,K,y,z,k, \\ &\hat{f}_t=f(t,y^1_t,\P_{y^1_t},z^1_t,\P_{z^1_t},G_t(k^1))-f(t,y^2_t,\P_{y^2_t},z^2_t,\P_{z^2_t},G_t(k^2)).
\end{align*}
It is easy to check that 
\begin{align*}
|\hat{Y}_t|&=\left|\E_t\left[\int_t^T \hat{f}_sds\right]+\hat{K}_T-\hat{K}_t\right|\\
&\leq \lambda \E_t\left[\int_0^T \left(|\hat{y}_s|+|\hat{z}_s|+\E[|\hat{y}_s|]+\E[|\hat{z}_s|]+\sup_{r\in[0,s]}|\hat{k}_r|\right)ds \right]+|\hat{K}_T|+|\hat{K}_t|.
\end{align*}
Applying the Doob's maximal inequality, we obtain that  
\begin{equation}\begin{split}\label{differY}
 \E\left[\sup_{t\in[0,T]}|\hat{Y}_t|^p\right]
 &\leq \beta \E\left[\left(\int_{0}^{T}|\hat{y}_s|+|\hat{z}_s|+\E[|\hat{y}_s|]+\E[|\hat{z}_s|]+\sup_{r\in[0,s]}|\hat{k}_r|ds\right)^p\right]+\beta \sup_{t\in[0,T]}|\hat{K}_t|^p.\\   
\end{split}
\end{equation}
Recalling the proof of Proposition \ref{pro-1} and \eqref{diffK}, we have
\begin{align*}
\sup_{t\in[0,T]}|\hat{K}_t|\leq  \beta \left\{\sup_{t\in[0,T]}|\hat{s}_t|+\sup_{(t,x)\in[0,T]\times\mathbb{R}}|\hat{l}(t,x)|\vee \sup_{(t,x)\in[0,T]\times\mathbb{R}} |\hat{r}(t,x)|\right\},
\end{align*}
where $\hat{s}_t=s^1_t-s^2_t$, $\hat{l}(t,x)=l^1(t,x)-l^2(t,x)$, $\hat{r}(t,x)=r^1(t,x)-r^2(t,x)$ and for $i=1,2$, 
\begin{align*}
&s^i_t=\E\left[\int_0^t f\left(s,y^i_s,\P_{y^i_s},z^i_s,\P_{z^i_s},G_s(k^i)\right)ds\right], \\ 
&l^i(t,x)=\E\left[L\left(t,\widetilde{Y}^i_t-\E[\widetilde{Y}^i_t]+x\right)\right],\\  &r^i(t,x)=\E\left[R\left(t,\widetilde{Y}^i_t-\E[\widetilde{Y}^i_t]+x\right)\right].
\end{align*}
Here, $\widetilde{Y}^i$ is the first component of the solution to the $\text{BSDE}$ with terminal value $\xi$ and constant driver $f^i$, $i=1,2$, where $$f^i_s=f\left(s,y_s^i,\P_{y_s^i},z_s^i,\P_{z^i_s},G_s(k^i)\right).$$ By Assumptions \ref{assLR} and \ref{assgk}, noting that $\widetilde{Y}^i_t=\E_t\left[\xi+\int_t^T f\left(s,y_s^i,\P_{y_s^i},z_s^i,\P_{z^i_s},G_s(k^i)\right)d s\right]$, we obtain that 
\begin{align*}
\left|l^1(t,x)-l^2(t,x)\right|&\leq C\E\left[\left|\left(\widetilde{Y}^1_t-\E[\widetilde{Y}^1_t]\right)-\left(\widetilde{Y}^2_t-\E[\widetilde{Y}^2_t]\right)\right|\right]\leq 2C\E[|\widetilde{Y}^1_t-\widetilde{Y}^2_t|]\\
&\leq \beta \E\left[\int_0^T |\hat{y}_s|+|\hat{z}_s|+\E[|\hat{y}_s|]+\E[|\hat{z}_s|]+\sup_{r\in[0,s]}|\hat{k}_r|d s\right ].
\end{align*}
Similar estimates hold for $|s^1_t-s^2_t|$ and $|r^1(t,x)-r^2(t,x)|$. 
Applying H{\"o}lder inequality, the above analysis yields that 
\begin{align}\label{differK}
\sup_{t\in[0,T]}|\hat{K}_t|^p\leq \beta\E\left[\left(\int_0^T |\hat{y}_s|+|\hat{z}_s|+\E[|\hat{y}_s|]+\E[|\hat{z}_s|]+\sup_{r\in[0,s]}|\hat{k}_r|ds\right)^p \right].
\end{align}
Combining Eqs. \eqref{differY} and \eqref{differK} yields that 
\begin{align*}
\E\left[\sup_{t\in[0,T]}|\hat{Y}_t|^p\right]\leq \beta\E\left[\left(\int_0^T |\hat{y}_s|+|\hat{z}_s|+\E[|\hat{y}_s|]+\E[|\hat{z}_s|]+\sup_{r\in[0,s]}|\hat{k}_r|d s \right)^p\right].
\end{align*}
Note that we have
\begin{align*}
\int_0^t \hat{Z}_s dB_s=\int_0^t\hat{f}_s ds+\hat{Y}_t-\hat{Y}_0+\hat{K}_t.
\end{align*}
Simple calculation yields that
\begin{align*}
\E\left[\sup_{t\in[0,T]}\left|\int_0^t \hat{Z}_sd B_s\right|^p\right]&\leq \beta\E\left[\sup_{t\in[0,T]}\left|\int_0^t\hat{f}_s ds\right|^p+\sup_{t\in[0,T]}|\hat{Y}_t|^p+\sup_{t\in[0,T]}|\hat{K}_t|^p\right]\\
&\leq \beta\E\left[\left(\int_0^T |\hat{y}_s|+|\hat{z}_s|+\E[|\hat{y}_s|]+\E[|\hat{z}_s|]+\sup_{r\in[0,s]}|\hat{k}_r| ds\right)^p\right],
\end{align*}
Combining with the B-D-G inequality yields that 
\begin{align*}
   \E\left[\left(\int_0^T |\hat{Z}_s|^2 d s\right)^{\frac{p}{2}}\right]&\leq \beta\E\left[\left(\int_0^T |\hat{y}_s|+|\hat{z}_s |+\E[|\hat{y}_s|]+\E[|\hat{z}_s|]+\sup_{r\in[0,s]}|\hat{k}_r| ds\right)^p\right].
\end{align*}
By the H\"{o}lder inequality and the Fubini theorem, we finally deduce that 
\begin{align*}
&\E\left[\sup_{s\in[0,T]}|\hat{Y}_s|^p+\left(\int_0^T |\hat{Z}_s|^2 d s\right)^{\frac{p}{2}}+\sup_{t\in[0,T]}|\hat{K}_t|^p\right]\\
\leq &\beta\E\left[\left(\int_0^T \left(|\hat{y}_s|+|\hat{z}_s|+\E[|\hat{y}_s|]+\E[|\hat{z}_s|]+\sup_{r\in[0,s]}|\hat{k}_r|\right)ds\right)^p\right]\\
\leq &\beta\E\left[\left(\int_0^T |\hat{y}_s|+|\hat{z}_s|+\E[|\hat{y}_s|]+\E[|\hat{z}_s|]ds\right)^p+\left(\int_0^T\sup_{r\in[0,s]}|\hat{k}_r|ds\right)^p\right]\\
\leq &\beta\max\left(T^{\frac{p}{2}},T^p\right)\E\left[\sup_{t\in[0,T]}|\hat{y}_t|^p+\left(\int_0^T|\hat{z}_s|^2 ds\right)^{\frac{p}{2}}\right]+ T^p \sup_{r\in[0,T]}|\hat{k}_r|^p\\
\leq &\beta\max\left(T^{\frac{p}{2}},T^p\right)\E\left[\sup_{t\in[0,T]}|\hat{y}_t|^p+\left(\int_0^T|\hat{z}_t|^2 dt\right)^{\frac{p}{2}}+\left(\sup_{t\in[0,T]}|\hat{k}_t|\right)^p\right].
\end{align*}
Therefore, we can choose an appropriate  $T \in(0,\varepsilon]$ such that
$$
\begin{aligned}
  &\left\|Y^{1}-Y^{2}\right\|_{\mathcal{S}^{p}}+\left\|Z^{1}-Z^{2}\right\|_{\mathcal{H}^{p}}+\sup_{t\in[0,T]}\left|K_{t}^{1}-K_{t}^{2}\right|\\
  &\leq \beta\max\left(T^{\frac{p}{2}},T^p\right) \left(\left\|y^1-y^2\right\|_{\mathcal{S}^{p}}+\left\|z^1-z^2\right\|_{\mathcal{H}^{p}}+\sup_{t\in[0,T]}\left|k_t^1-k_t^2\right|\right). 
\end{aligned}
$$
The proof is complete.
\end{proof}

Now,  we are ready to give the proof of Theorem \ref{thlip}. 

\begin{proof}[Proof of Theorem \ref{thlip}] First, for a given $k\in BV[0,T]$, according to Theorem 3.5 in \cite{LS},  the following MFBSDE with double mean reflections
 \begin{equation}\label{MFBSDEDMRkL}
\begin{cases}
Y_t=\xi+\int_t^T f(s,Y_s,\mathbf{P}_{Y_s},Z_s,\mathbf{P}_{Z_s},G_s(k))ds-\int_t^T Z_s dB_s+K_T-K_t, \\
\mathbf{E}[L(t,Y_t)]\leq 0\leq \mathbf{E}[R(t,Y_t)], \\
K_t=K^R_t-K^L_t,\ K^R,K^L\in I[0,T],\\
\int_0^T \mathbf{E}[R(t,Y_t)]dK_t^R=\int_0^T \mathbf{E}[L(t,Y_t)]dK^L_t=0.\\
\end{cases}
\end{equation}
admits a unique solution $(Y,Z,K)\in\mathcal{S}^p \times \mathcal{H}^{p} \times BV[0,T]$. 

Now, let $N$ be the smallest integer such that $N\geq \frac{T}{\varepsilon}$, where $\varepsilon$ is be chosen as Lemma \ref{le3-2}. For $k^i\in BV[0,T],\ i=1,2$, let $(Y^i,Z^i,K^i)$ be the unique solution to \eqref{MFBSDEDMRkL} with data $k^i$. According to Lemma \ref{le3-2}, for any $1\leq i\leq N$, 
$$
\begin{aligned}
  &\left\|Y^{1}-Y^{2}\right\|_{\mathcal{S}^{p}{[(i-1)\varepsilon,i\varepsilon \wedge T]}}+\left\|Z^{1}-Z^{2}\right\|_{\mathcal{H}^{p}[(i-1)\varepsilon,i\varepsilon \wedge T]}+\sup_{t\in[(i-1)\varepsilon,i\varepsilon \wedge T]}\left|K_{t}^{1}-K_{t}^{2}\right|\\
  &\leq \beta\max\left(\varepsilon^{\frac{p}{2}},\varepsilon^p\right) \left(\left\|Y^{1}-Y^{2}\right\|_{\mathcal{S}^{p}[(i-1)\varepsilon,i\varepsilon \wedge T]}+\left\|Z^{1}-Z^{2}\right\|_{\mathcal{H}^{p}[(i-1)\varepsilon,i\varepsilon \wedge T]}+\sup_{t\in[(i-1)\varepsilon,i\varepsilon \wedge T]}\left|k_t^1-k_t^2\right|\right).
\end{aligned}
$$
Therefore, we have 
$$\sup_{t\in[(i-1)\varepsilon,i\varepsilon \wedge T]}\left|K_t^1-K_t^2\right|\leq \max_{1\leq i \leq N}\left(\sup_{t\in[(i-1)\varepsilon,i\varepsilon \wedge T]}\left|K_t^1-K_t^2\right|\right)\leq \beta\max\left(\varepsilon^{\frac{p}{2}},\varepsilon^p\right)\sup_{t\in[0,T]}\left|k_t^1-k_t^2\right|.$$
By the principle of standard contraction mapping, the doubly mean reflected MFBSDE \eqref{nonlinearyz} with resistance  admits a unique solution $(Y,Z,K)\in\mathcal{S}^p \times \mathcal{H}^{p} \times BV[0,T].$   
\end{proof}

\section{Doubly mean reflected MFBSDEs with nonlinear resistance: the quadratic case}	
In this section, we investigate the well-posedness of doubly mean reflected \text{MFBSDE} with nonlinear resistance \eqref{nonlinearyz} when its generator $f$ has quadratic growth  and the terminal value is bounded. More precisely, we propose the  following assumption.
\begin{assumption}\label{assqua}
\begin{itemize}
      \item[(i)]  The terminal value $\xi \in \mathcal{L}^{\infty}_T $ satisfies $\mathbf{E}[L(T,\xi)]\leq 0 \leq \mathbf{E}[R(T,\xi)]$, i.e. there exist some constants $H_1$ such that $\left \| \xi  \right \|_{\mathcal{L}^\infty}\leq H_1$; 
      \item[(ii)] The process $\left\{f\left(t, 0, \delta_0, 0,\delta_0,0 \right)\right\}_{t\in[0,T]}$ is uniformly bounded by some constant $H_2$, $\mathbb{P}$-a.s. There exist some positive constants $\lambda$ and $\alpha\in[0,1)$ such that, $\mathbb{P}$-a.s., for any $t \in[0, T]$, $ y_1, y_2 , k_1, k_2\in \mathbb{R}$, $\mu_1, \mu_2 \in \mathcal{P}_1(\mathbb{R})$, $z_1, z_2 \in \mathbb{R}$ and $\nu_1, \nu_2 \in \mathcal{P}_1(\mathbb{R})$, we have
	$$
    \begin{aligned}
		& \left|f\left(t, y_1, \mu_1, z_1,\nu_1, k_1\right)-f\left(t, y_2, \mu_2, z_2, \nu_2, k_2\right)\right| \\
		\leq &\lambda\left(\left|y_1-y_2\right|+\left(1+\left|z_1\right|+\left|z_2\right|\right)\left|z_1-z_2\right|+\left|k_1-k_2\right|\right)\\
        & +\lambda \mathcal{W}_1\left(\mu_1,\mu_2\right)+\lambda \left(1+\left(\mathcal{W}_1\left(\nu_1,\delta_0\right)\right)^{\alpha}+\left(\mathcal{W}_1\left(\nu_2,\delta_0\right)\right)^{\alpha}\right)\mathcal{W}_1\left(\nu_1,\nu_2\right).
	\end{aligned}
	$$ 
   \end{itemize}
 \end{assumption}
 
\begin{remark}
\begin{itemize}
    \item[(i)] We will take $H=H_1 \vee H_2$ in the following text.
    \item[(ii)]
 Applying the H\"{o}lder inequality, we have $$W_1(\mu,\nu)\leq W_2(\mu,\nu).$$
\end{itemize}
\end{remark}

First, we state the main result in this section.
\begin{theorem}\label{globalMFBSDE}
Let Assumptions \ref{assLR}, \ref{assgk}, \ref{assqua} hold. Suppose that for any $(t,y,\mu,\nu,k)\in [0,T]\times\mathbb{R}\times\mathcal{P}_1(\mathbb{R})\times\mathcal{P}_1(\mathbb{R})\times\mathbb{R}$, there exists a constant $\tilde{H}>0$, such that $\left|f(t,y,\mu,0,\nu,k)\right|\leq \tilde{H}$. Then, for any time horizon $[0,T]$, doubly mean reflected MFBSDE  with nonlinear resistance \eqref{nonlinearyz} has a unique solution $ (Y,Z,K)\in\mathcal{S}^{\infty}\times BMO \times BV[0,T].$     
\end{theorem}

The construction of the solution to doubly mean reflected MFBSDE with nonlinear resistance needs a contraction mapping argument. Due to the fact that $f$ has quadratic growth in $z$, the coefficient depends on the uniform estimate of $Y$ (see $\tilde{A}$ in Eq. \eqref{contraction}, where $\tilde{A}$ appears in \eqref{BA}). Therefore, we first provide the following result concerning the uniform estimate of the first component of the solution to a doubly mean reflected MFBSDE with an additional bounded assumption for the generator.
\begin{lemma}\label{lem4.6}
Let Assumptions \ref{assLR}, \ref{assgk}, \ref{assqua} hold. Suppose that the generator $f$ does not depend on $k$, and for any $(t,y,\mu,\nu)\in [T-h,T]\times\mathbb{R}\times\mathcal{P}_1(\mathbb{R})\times\mathcal{P}_1(\mathbb{R})$, there exists a constant $\hat{H}$ such that $\left|f(t,y,\mu,0,\nu)\right|\leq \hat{H}$. Moreover, assume that the doubly mean reflected MFBSDE \eqref{nonlinearyz} has a local solution $(Y,Z,K)\in \mathcal{S}^{\infty}[T-h,T]\times BMO[T-h,T]\times BV[T-h,T]$ for some $0<h\leq T$. Then, there exists a constant $\bar{H}$ depending only on $C,c,H,\hat{H},\lambda, T,M$ such that 
$$\left\|Y\right\|_{\mathcal{S}^{\infty}[T-h,T]}\leq \bar{H}.$$
\end{lemma}
\begin{proof}
 First, according to Theorem 3.11 in \cite{HHTW}, the following  quadratic MFBSDE 
\begin{equation}\label{yz}
 \bar{Y}_{t}=\xi+\int_t^{T} f(s,Y_s,\P_{Y_s},\bar{Z}_{s}, \P_{\bar{Z}_{s}})d s-\int_t^{T} \bar{Z}_{s} dB_s , 
 \end{equation}
 admits a unique solution $(\bar{Y},\bar{Z})\in\mathcal{S}^{\infty}[T-h,T]\times BMO[T-h,T]$. Therefore, similar as the proof of Lemma \ref{4.1.1}, $Y$ has the following representation:
$$Y_{t}=\bar{Y}_{t}+K_{T}-K_{t}=\bar{Y}_{t}+\bar{K}_{T-t},\ \forall t\in[T-h,T],$$
and $\bar{K}$ is the second component of the solution to the Skorokhod problem $\mathbb{SP}_{\bar{l}}^{\bar{r}}(\bar{s})$, which satisfies $K_t=\bar{K}_{h}-\bar{K}_{T-t}$. 
Here for each $t\in[T-h,T]$, 
\begin{equation}\label{salry}\begin{split}
&\bar{s}_t=\E\left[\bar{Y}_{T-t}\right],\\
&\bar{l}(t,x):=\E\left[L(T-t,\bar{Y}_{T-t}-\E\left[\bar{Y}_{T-t}\right]+x)\right],\\
&\bar{r}(t,x):=\E\left[R(T-t,\bar{Y}_{T-t}-\E\left[\bar{Y}_{T-t}\right]+x)\right].\\
\end{split}\end{equation}
Furthermore, 
\begin{equation}\begin{split}
\left\|Y\right\|_{\mathcal{S}^{\infty}[T-h,T]}&\leq\left\|\bar{Y}\right\|_{\mathcal{S}^{\infty}[T-h,T]}+\sup_{t\in[T-h,T]}\left|\bar{K}_{T-t}\right|\\
&\leq\left\|\bar{Y}\right\|_{\mathcal{S}^{\infty}[T-h,T]}+\sup_{t\in[T-h,T]}\left|\bar{K}_{T-t}-\bar{K}^0_{T-t}\right|+\sup_{t\in[T-h,T]}\left|\bar{K}^0_{T-t}\right|,\\
\end{split}   
\end{equation}
where $\bar{K}^0$ is the second component of the solution to the Skorokhod problem $\mathbb{SP}^{\bar{r}^0}_{\bar{l}^0}(\bar{s}^0)$ and $\bar{s}^0$, $\bar{r}^0$, $\bar{l}^0$ are the same as in  \eqref{slr}. In view of Proposition \ref{continuity}, we derive that 
$$
\begin{aligned}
\sup_{t \in[T-h, T]}\left|\bar{K}_{T-t}-\bar{K}^0_{T-t}\right|&=\sup_{t \in[0, h]}\left|\bar{K}_{t}-\bar{K}^0_{t}\right|\\
&\leq \frac{3C}{c}\sup_{t \in[T-h, T]}\mathbf{E}\left[\left|\bar{Y}_{t}\right|\right]+\frac{C}{c}\mathbf{E}\left[\left|\xi\right|\right]\\
&\leq \frac{3C}{c}\left\|\bar{Y}\right\|_{\mathcal{S}^{\infty}[T-h, T]}+\frac{C}{c}H.\\  
\end{aligned}		
$$
Recalling \eqref{barK0}, we have
\begin{equation*}
 \sup_{t\in[T-h,T]}\left|\bar{K}^0_{T-t}\right|\leq \sup_{t\in[0,T]}\left|\bar{K}^0_{t}\right|\leq \frac{2C}{c}\mathbf{E}[\left|\xi\right|]+\frac{M}{c}\leq \frac{2C}{c}H+\frac{M}{c}.   
\end{equation*}
According to Proposition 2.2 in \cite{BE}, it is easy to check that for each $t\in[T-h,T]$,
$$\left|\bar{Y}_t\right|\leq(H+\hat{H}T)e^{\lambda T}:=\bar{H}^1.$$
Therefore,
\begin{equation}\label{barH}
\left\|Y\right\|_{\mathcal{S}^{\infty}[T-h,T]}\leq(1+\frac{3C}{c})\bar{H}^1+\frac{3C}{c}H+\frac{M}{c}:=\bar{H}, \ \forall t\in[T-h,T],   
\end{equation}
which is the desired result.
\end{proof}

Next, when $T$ is sufficiently small, we show that the quadratic MFBSDE with double mean reflections and nonlinear resistance \eqref{nonlinearyz} admits a unique solution  by utilizing a contraction mapping method. For this purpose, given $
\left(y, z, k\right)\in \mathcal{S}^{\infty} \times BMO\times BV[0,T]$, consider the following quadratic MFBSDE with double mean reflections  
\begin{equation}\label{nonlinearyz-1}
\begin{cases}
Y_t^{y,z,k}=\xi+\int_t^T f(s,y_s,\P_{y_s},Z_s^{y,z,k},\P_{z_s},G_s(k))ds-\int_t^T Z_s^{y,z,k} dB_s+K_T^{y,z,k}-K_t^{y,z,k},\\
\E[L(t,Y_t^{y,z,k})]\leq 0\leq \E[R(t,Y_t^{y,z,k})], \\
K_t^{y,z,k}=K^{y,z,k,R}_t-K^{y,z,k,L}_t,\ K^{y,z,k,R},K^{y,z,k,L}\in I[0,T],\\
\int_0^T \E[R(t,Y_t)]dK^{y,z,k,R}_t=\int_0^T \E[L(t,Y_t)]dK^{y,z,k,L}_t=0.
\end{cases}
\end{equation}  
Then, under Assumptions \ref{assLR}, \ref{assgk}, \ref{assqua}, it follows from Theorem 4.2 in \cite{LS} that the above equation has a unique solution $\left(Y^{y,z,k}, Z^{y,z,k}, K^{y,z,k}\right)\in \mathcal{S}^{\infty} \times BMO\times BV[0,T]$.     
Now, we define the map 
\begin{align*}
\Gamma:\mathcal{S}^{\infty} \times BMO\times BV[0,T]&\rightarrow \mathcal{S}^{\infty} \times BMO \times BV[0,T],  \\
\Gamma\left(y,z,k\right)&:=\left(Y^{y,z,k}, Z^{y,z,k}, K^{y,z,k}\right).
\end{align*}
We set 
\begin{equation}\label{BA}
    \mathscr{B}_{\tilde{A}} := \left\{ \left(y, z, k\right)\in \mathcal{S}^{\infty} \times BMO \times BV[0,T]:\left \| y \right \|_{\mathcal{S}^{\infty}}\leq \tilde{A},  \left \| z \right \|_{BMO}\leq \tilde{A}, \sup_{0 \leq t \leq T}\left|k_t\right|\leq \tilde{A}\right\},
\end{equation}
where $\tilde{A} \geq \tilde{A}_0$ and
\begin{equation}\label{A0}
 \tilde{A}_0=(3+\frac{24C}{c})H+\frac{2M}{c}+(\frac{2H}{\lambda}+\frac{1}{3})e^{9\lambda H}.   
\end{equation}

	\begin{theorem}\label{thmbddterminal}
	  Let Assumptions \ref{assLR}, \ref{assgk}, \ref{assqua} hold. Then, there exists a constant $\hat{\delta}^{\tilde{A}}:=\hat{\delta}^{\tilde{A}}(\tilde{A}, c, C, H, \lambda)>0$ such that for any $T\in(0,\hat{\delta}^{\tilde{A}}]$, the quadratic MFBSDE \eqref{nonlinearyz} with double mean reflections and nonlinear resistance  admits a unique solution $(Y, Z, K) \in \mathcal{S}^{\infty} \times BMO \times BV[0,T]$. Moreover, we have  
     $$\left \| Y \right \|_{\mathcal{S}^{\infty}}\leq \tilde{A}, \  \left \| Z \right \|_{BMO}\leq \tilde{A} \text{ and } \sup_{0 \leq t \leq T}\left|K_t\right|\leq \tilde{A}.$$
	\end{theorem}

The proof of Theorem \ref{thmbddterminal} needs the following lemma. 
\begin{lemma}\label{4.1.1}
Let Assumptions \ref{assLR}, \ref{assgk}, \ref{assqua} hold. Then, there exists a constant $\delta^{\tilde{A}}>0$ depending only on $\tilde{A}, H, \lambda, C, c$ such that for any $T\in(0,\delta^{\tilde{A}}]$, $\Gamma(\mathscr{B}_{\tilde{A}})\subset \mathscr{B}_{\tilde{A}}$. Furthermore, for any $T\in(0,\hat{\delta}^{\tilde{A
}}]$, $\Gamma$ is contractive.
\end{lemma}

\begin{proof}
\textbf{Step 1.} We first show that $\Gamma(\mathscr{B}_{\tilde{A}})\subset \mathscr{B}_{\tilde{A}}$. For notional simplification, set
$$Y_t:=Y^{y,z,k}_t, Z_t:=Z^{y,z,k}_t,K:=K^{y,z,k}_t,f^{y,z,k}(t,z):=f(t,y_t,\P_{y_t},z,\P_{z_{t}},G_{t}(k)).$$
Similar with the proof of Proposition \ref{pro-1}, we have 
 \begin{equation}\label{YmZm}
   Y_{t}=\bar{Y}_{t}+K_{T}-K_{t},\ Z_t=\bar{Z}_{t},
 \end{equation}
 where $(\bar{Y},\bar{Z})\in \mathcal{S}^{\infty}\times \text{BMO}$ is the unique solution to quadratic BSDE with terminal condition $\xi$ and the generator $f^{y,z,k}$, and $K$ is the second component of the solution to the backward Skorokhod problem $\mathbb{BSP}_{l}^{r}(s,a)$. Here, for each $t\in[0,T]$,  
\begin{equation}\begin{split}
&s_t=\E\left[\bar{Y}_{0}-\bar{Y}_{t}\right],\ a=\mathbf{E}[\xi],\\
&l(t,x)=\E\left[L(t,\bar{Y}_{t}-\E\left[\bar{Y}_{t}\right]+x)\right],\\
&r(t,x)=\E\left[R(t,\bar{Y}_{t}-\E\left[\bar{Y}_{t}\right]+x)\right].\\
\end{split}\end{equation}
In addition, according to Proposition 4.2 in \cite{L} and Remark \ref{rem2.7}, we know that the solution $K$ of \eqref{nonlinearyz-1} has the following representation 
\begin{equation}\label{kbark}
K_t=\bar{K}_{T}-\bar{K}_{T-t},
\end{equation}
where $\bar{K}$ is the second component of the solution to the Skorokhod problem $\mathbb{SP}_{\bar{l}}^{\bar{r}}(\bar{s})$.
Here for each $t\in[0,T]$,  
\begin{equation}\label{salr}\begin{split}
\bar{s}_t&=\E\left[\bar{Y}_{T-t}\right],\\
\bar{l}(t,x)&=\E\left[L(T-t,\bar{Y}_{T-t}-\E\left[\bar{Y}_{T-t}\right]+x)\right],\\
\bar{r}(t,x)&=\E\left[R({T-t},\bar{Y}_{T-t}-\E\left[\bar{Y}_{T-t}\right]+x)\right].
\end{split}\end{equation}
Combining Eqs. \eqref{YmZm} and \eqref{kbark}, we obtain that  
$$
Y_t=\bar{Y}_t+K_T-K_t=\bar{Y}_t+\bar{K}_{T-t},  \ t\in[0,T].
$$
Consequently, we have
\begin{equation}\label{re-1}
\begin{split}
 \left \| Y \right \|_{\mathcal{S}^{\infty}}&\leq \left \| \bar{Y} \right \|_{\mathcal{S}^{\infty}}+\sup_{t\in[0,T]}\left|\bar{K}_t\right|\\
 &\leq \left \| \bar{Y} \right \|_{\mathcal{S}^{\infty}}+\sup_{t\in[0,T]}\left|\bar{K}_t-\bar{K}^0_t\right|+\sup_{t\in[0,T]}\left|\bar{K}^0_t\right|,   
\end{split}   
\end{equation}
where $\bar{K}^0$ is the second component of the solution to the Skorokhod problem $\mathbb{SP}^{\bar{r}^0}_{\bar{l}^0}(\bar{s}^0)$. Here, for each $t\in[0,T]$, 
\begin{equation}\label{slr}
\bar{s}^0_t=\E[\xi],\ \bar{l}^0(t,x)=\E\left[L(T-t,x)\right], \ \bar{r}^0(t,x)=\E\left[R(T-t,x)\right].
\end{equation}
Let $\phi_t,\psi_t$ be two constants satisfying the following equations, respectively,
\begin{equation*}\begin{split}
&\bar{l}^0(t,\bar{s}_t^0+ \phi_t)=0, \ \bar{r}^0(t,\bar{s}_t^0+ \psi_t)=0.\\
\end{split}\end{equation*} 
By Proposition \ref{SP}, we have 
\begin{equation}\label{barK0}\begin{split}
\sup_{t\in[0,T]}\left|\bar{K}^{0}_{t}\right|
&\leq \sup_{t\in[0,T]}\left|\phi_{t}\right|+\sup_{t\in[0,T]}\left|\psi_{t}\right|\\
&\leq \frac{1}{c} \sup_{t\in[0,T]}\left|\bar{l}^0(t,\E[\xi]+ \phi_t)-\bar{l}^0(t,\E[\xi])\right|+ \frac{1}{c} \sup_{t\in[0,T]}\left|\bar{r}^0(t,\E[\xi]+ \psi_t)-\bar{r}^0(t,\E[\xi])\right|\\
&\leq \frac{1}{c} \sup_{t\in[0,T]}\left|\E[L(T-t, \E[\xi])-L(T-t,0)]\right|+ \frac{1}{c} \sup_{t\in[0,T]}\left|\E[L(T-t,0)]\right|\\
&\quad+\frac{1}{c} \sup_{t\in[0,T]}\left|\E[R(T-t, \E[\xi])-L(T-t,0)]\right|+ \frac{1}{c} \sup_{t\in[0,T]}\left|\E[R(T-t,0)]\right|\\
&\leq \frac{2C}{c}\mathbf{E}[\left|\xi\right|]+\frac{M}{c}.\\
\end{split}\end{equation}
 In view of Proposition \ref{continuity}, it is easy to check that 
\begin{equation}\label{barKK0}
	\sup_{t \in[0, T]}\left|\bar{K}_{t}-\bar{K}^0_{t}\right|
	\leq \frac{3C}{c}\sup_{t \in[0, T]}\mathbf{E}[|\bar{Y}_t|]+\frac{C}{c}\E[|\xi|].
\end{equation}
Let us recall the following estimates, which can be found in the proof of Lemma 2.10 in \cite{LP} (see Eqs. (13) and (14) in \cite{LP})  
\begin{equation}\label{barYZ}\begin{split}
\left\|\bar{Y}\right\|_{\mathcal{S}^{\infty}}
&\leq (1+T)H+3\lambda T \tilde{A}+\lambda\sqrt{T} \tilde{A}+\lambda T^{\frac{1-\alpha}{2}} \tilde{A}^{1+\alpha},\\
\left\|Z\right\|_{BMO}^2
&\leq \frac{2 e^{3\lambda\left\|\bar{Y}\right\|_{\mathcal{S}^{\infty}}}}{3}\left(\frac{(1+T)H}{\lambda}+3 T \tilde{A}+\frac{T}{2} +\sqrt{T}\tilde{A}+T^{\frac{1-\alpha}{2}}\tilde{A}^{1+\alpha}\right).
\end{split}\end{equation}

Define 
$$\delta^{\tilde{A}}:=\min\left(\frac{H}{9\lambda \tilde{A}},\frac{H^2}{9\lambda^2 \tilde{A}^2}, \left(\frac{H}{3\lambda \tilde{A}^{1+\alpha}}\right)^{\frac{2}{1-\alpha}}\right).$$
Therefore, we  can derive that for each $T\in(0,\delta^{\tilde{A}}]$, 
$$
\begin{aligned}
 &\left \| \bar{Y} \right \|_{\mathcal{S}^{\infty}}\leq 3H,\ \ \ \ \left \| Z \right \|_{BMO} \leq \sqrt{\left(\frac{2H}{\lambda}+\frac{1}{3}\right)e^{9\lambda H}}\leq \tilde{A},\\
 &\left \| Y \right \|_{\mathcal{S}^{\infty}}\leq \left(1+\frac{4C}{c}\right)3H+\frac{M}{c}\leq \tilde{A}.
\end{aligned} 
$$
Meanwhile, recalling \eqref{barK0} and \eqref{barKK0}, we have 
$$
\begin{aligned}
 \sup_{t\in[0,T]}\left|K_t\right|&=\sup_{t\in[0,T]}\left|\bar{K}_T-\bar{K}_{T-t}\right|\leq2\sup_{t\in[0,T]}\left|\bar{K}_t\right|\\
 &\leq2\sup_{t \in[0, T]}\left|\bar{K}_{t}-\bar{K}^0_{t}\right|+2\sup_{t \in[0, T]}\left|\bar{K}^0_{t}\right|\\
 &\leq \frac{24C}{c}H+\frac{2M}{c}\leq \tilde{A}.\\   
\end{aligned} 
$$

\textbf{Step 2.} Now, we show that $\Gamma$ is a contraction mapping.  For each $i=1,2$, we denote 
$$
(Y^i,Z^i,K^i)=\Gamma(y^i,z^i,k^i), \ f^i(t,z)=f(t,y^i_{t},\P_{y^i_{t}},z,\P_{z^i_{t}},G_{t}(k^i)).
$$
Similar as the analysis in Step 1, for $i=1,2$, we have     
\begin{equation}\label{re-2}
Y^{i}_t=\bar{Y}^{i}_{t}+K^{i}_{T}-K^{i}_t, \quad \forall t\in[0,T],   
\end{equation} 
where $(\bar{Y}^i,\bar{Z}^i)\in\mathcal{S}^{\infty}\times BMO$ is the solution to BSDE with terminal condition $\xi$ and the generator $f^i$, and $K^i$ is the second component of the solution to the backward Skorokhod problem $\mathbb{BSP}_{l^i}^{r^i}(s^i,a^i)$. Here, for each $t\in[0,T]$,  
\begin{equation}\label{bsalri}\begin{split}
&s^i_t=\E\left[\bar{Y}^i_{0}-\bar{Y}^i_{t}\right],\ a^i=\mathbf{E}[\xi],\\
&l^i(t,x)=\E\left[L(t,\bar{Y}^i_{t}-\E\left[\bar{Y}^i_{t}\right]+x)\right],\\
&r^i(t,x)=\E\left[R(t,\bar{Y}^i_{t}-\E\left[\bar{Y}^i_{t}\right]+x)\right].\\
\end{split}\end{equation}
Therefore, we have
\begin{equation}\begin{split}
   \left\|Y^1-Y^2\right\|_{\mathcal{S}^{\infty}}
   &\leq \left\|\bar{Y}^1-\bar{Y}^2\right\|_{\mathcal{S}^{\infty}}+2\sup_{t\in[0,T]}\left|K^1_t-K_t^2\right|.
\end{split}   
\end{equation}
Recall the following estimates from the proof 
of Lemma 2.11 in \cite{LP} (see Eqs. (23) and (24) in \cite{LP})
$$
\begin{aligned}
   \left\|\bar{Y}^1-\bar{Y}^2\right\|_{\mathcal{S}^{\infty}} &\leq \lambda \left(2T\left\|y^1-y^2\right\|_{\mathcal{S}^{\infty}}+T\sup_{t\in[0,T]}\left|k_t^1-k_t^2\right|+\sqrt{3T+6T^{1-\alpha}\tilde{A}^{2\alpha}}\left\|z^1-z^2\right\|_{BMO}\right),\\
   \left\|Z^1-Z^2\right\|_{BMO}&\leq 2\lambda \sqrt{1+12\lambda^2+24 \lambda^2\tilde{A}^2}\left(2T\left\|y^1-y^2\right\|_{\mathcal{S}^{\infty}}+T\sup_{t\in[0,T]}\left|k_t^1-k_t^2\right|\right. \\
   &\quad \left.+\sqrt{3T+6T^{1-\alpha}\tilde{A}^{2\alpha}}\left\|z^1-z^2\right\|_{BMO}\right).\\
\end{aligned}
$$
In view of Proposition \ref{continuous}, we obtain that
$$
\begin{aligned}
\sup_{t \in[0, T]}\left|K^1_{t}-K^2_{t}\right|
\leq \frac{12C}{c}\sup_{t \in[0, T]}\mathbf{E}\left[\left|\bar{Y}^1_t-\bar{Y}^2_t\right|\right]\leq \frac{12C}{c}\left\|\bar{Y}^1-\bar{Y}^2\right\|_{\mathcal{S}^{\infty}}.
\end{aligned}		
$$
Therefore, we derive 
\begin{equation}\label{contraction}\begin{split}
   &\left\|Y^1-Y^2\right\|_{\mathcal{S}^{\infty}}+\left\|Z^1-Z^2\right\|_{BMO}+\sup_{t \in[0, T]}\left|K^1_{t}-K^2_{t}\right|\\
   \leq &\lambda \hat{A} \left(2T\left\|y^1-y^2\right\|_{\mathcal{S}^{\infty}}+\sqrt{3T+6T^{1-\alpha}\tilde{A}^{2\alpha}}\left\|z^1-z^2\right\|_{BMO}+T\sup_{t\in[0,T]}\left|k^1_t-k^2_t\right|\right),\\
\end{split}   
\end{equation}
where $$\hat{A}=\left(1+\frac{36C}{c}+2\sqrt{1+12\lambda^2+24\lambda^2\tilde{A}^2}\right).$$

Now we define
\begin{equation}\label{delta}
  \hat{\delta}^{\tilde{A}}:=\min\left(\frac{1}{4\hat{A}\lambda}, \frac{1}{12\hat{A}\lambda}, \left(\frac{1}{24\tilde{A}^{2\alpha}}\hat{A}^2\lambda^2\right)^{\frac{1}{1-\alpha}},\delta^{\tilde{A}}\right). 
\end{equation}
It is straightforward to check that for any $T\in(0,\hat{\delta}^{\tilde{A}}]$, we have
$$\begin{aligned}
  &\left\|Y^1-Y^2\right\|_{\mathcal{S}^{\infty}}+\left\|Z^1-Z^2\right\|_{BMO}+\sup_{t \in[0, T]}\left|K^1_{t}-K^2_{t}\right|\\
   \leq&\frac{1}{2}\left(\left\|y^1-y^2\right\|_{\mathcal{S}^{\infty}}+\left\|z^1-z^2\right\|_{BMO}+\sup_{t\in[0,T]}\left|k^1_t-k^2_t\right|\right).\\  
\end{aligned}$$
The proof is complete.
\end{proof}

Now, we are ready to give the proof of Theorem $\ref{thmbddterminal}$.

\begin{proof}[Proof of Theorem \ref{thmbddterminal}] We take $\tilde{A}\geq \tilde{A}_0$ and choose $\hat{\delta}^{\tilde{A}}$ as $(\ref{delta})$. For $T\in(0,\hat{\delta}^{\tilde{A}}]$, we define $(Y^0,Z^0,K^0)=(0,0,0)$ and $(Y^{i+1},Z^{i+1},K^{i+1})=(Y^{Y^{i},Z^{i},K^{i}},Z^{Y^{i},Z^{i},K^{i}},K^{Y^{i},Z^{i},K^{i}})$ by recurrence. 
For convenience, we denote 
$$
\begin{aligned}
  f^n_t&:=f(t,Y^{n-1}_{t},\mathbf{P}_{Y^{n-1}_{t}},Z_t^{n},\mathbf{P}_{Z_t^{n-1}},G_{t}(K^{n-1})),\\
  f_t&:=f(t,Y_{t},\mathbf{P}_{Y_{t}},Z_t,\mathbf{P}_{Z_t},G_{t}(K)).
\end{aligned}
$$
By the proof of Lemma \ref{4.1.1}, for each $n\geq 1$, $K^n$ is the second component of the solution to $\mathbb{BSP}^{r^n}_{l^n}(s^n,a^n)$, where for each $t\in[0,T]$,
\begin{equation}\label{sarln}\begin{split}
& s_t^{n}=\mathbf{E}\left[\int_{0}^{t} f_s^n ds\right],\ l^n(t,x)=\E\left[L(t,\widetilde{Y}^n_t-\E[\widetilde{Y}^n_t]+x)\right],\\
& a^n=\mathbf{E}[\xi], \ r^n(t,x)=\E\left[R(t,\widetilde{Y}^n_t-\E[\widetilde{Y}^n_t]+x)\right]. \end{split}
 \end{equation}
and $\widetilde{Y}^n$ is the first component of the solution to  the following BSDE:
\begin{equation*}
\widetilde{Y}^n_t=\xi+\int_{t}^{T} f_s^nds-\int_{t}^{T}\widetilde{Z}^n_s dB_s.    
\end{equation*}
It follows from Lemma \ref{4.1.1} that there exists $(Y,Z,K)\in \mathscr{B}_{\tilde{A}}$ such that 
\begin{equation}\label{YnY}
\left\|Y^n-Y\right\|_{\mathcal{S}^{\infty}}\to 0,\ \left\|Z^n-Z\right\|_{BMO}\to 0,\ \sup_{t\in[0,T]}|K_t^n-K_t|\to 0.    
\end{equation}

In order to show that the triple $(Y,Z,K)$ is indeed a solution to MFBSDE with nonlinear resistance, it suffices to prove the Skorokhod conditions hold. The proof is similar to the one for Theorem 4.4 in \cite{LS}. For readers' convenience, we give a short proof here. For this purpose, let $(Y',Z',K')$ be the solution to the doubly mean reflected BSDE on $[0,T]$ with terminal value $\xi$, constant coefficient $\{f(s,Y_s,\P_{Y_s},Z_s,\P_{Z_s},G_s(K))\}_{s\in[0,T]}$ and loss functions $L,R$. It suffices to prove that $Y'\equiv Y$ and $K'\equiv K$. Let $(\widetilde{Y},\widetilde{Z})$ be the solution to the BSDE on $[0,T]$ with terminal value $\xi$ and constant coefficient $\{f(s,Y_s,\P_{Y_s},Z_s,\P_{Z_s},G_s(K))\}_{s\in[0,T]}$. According to Proposition \ref{pro-1}, we have
\begin{equation*}
Y'_{t}=\widetilde{Y}_t+K'_{T}-K'_{t}=\xi+\int_{t}^{T} f_s ds-\int_{t}^{T}\widetilde{Z}_s dB_s+K'_{T}-K'_{t},    
\end{equation*}
where $K'$ is the second component of the solution to $\mathbb{BSP}^{r}_{l}(s,a)$. 
Here, for any $(t,x)\in[0,T]\times \mathbb{R}$,
\begin{equation}\label{lrsaL}
\begin{split}
& s_t=\E\left[\int_{0}^t f_s ds\right], \ l(t,x)=\E\left[L(t,\widetilde{Y}_t-\E[\widetilde{Y}_t]+x)\right],\\
& a=\E[\xi], \ r(t,x)=\E\left[R(t,\widetilde{Y}_t-\E[\widetilde{Y}_t]+x)\right].
\end{split}
\end{equation}
Therefore, we have
\begin{align*}
   \sup_{t\in[0,T]}|s_{t}^{n}-s_{t}|\rightarrow 0, \ n\rightarrow \infty.
\end{align*}
In addition, applying Assumptions \ref{assLR}, \ref{assqua} and B-D-G inequality , we obtain that 
$$
\begin{aligned}
	\bar{L}^n_{T}&:= \sup_{(t,x)\in[0,T]\times \mathbb{R}}\left|l^n(t,x)-l(t,x)\right|\leq 2C\sup_{t\in[0,T]}\mathbf{E}\left[\left|\widetilde{Y}^n_t-\widetilde{Y}_t\right|\right]\\
    &\leq 2C\sup_{t\in[0,T]}\mathbf{E}\left[\int_{t}^{T}\left|f_s^n-f_s\right|ds\right]+2C\sup_{t\in[0,T]}\mathbf{E}\left[\left|\int_{t}^{T}\tilde{Z}^n_s-\tilde{Z}_sdB_s\right|\right]\\
    &\leq 2\lambda \mathbf{E}\left[\int_{0}^{T}\left|Y^n_s-Y_s\right|ds\right]+\lambda \int_{0}^{T}\left|K^{n-1}_s-K_s\right|ds \\
    &\quad + \lambda \mathbf{E}\left[\int_{0}^{T}\left(1+\left|Z^n_s\right|+\left|Z_s\right|\right)\left|Z^n_s-Z_s\right|ds\right]\\
&\quad+\lambda \mathbf{E}\left[\int_{0}^{T}\left(1+\mathbf{E}\left[\left|Z^n_s\right|\right]+\mathbf{E}\left[\left|Z_s\right|\right]\right)\left|Z^n_s-Z_s\right|ds\right]\\
&\quad+\Lambda \mathbf{E}\left[\int_{0}^{T}\left(\tilde{Z}^n_s-\tilde{Z}_s\right)^2ds\right]^{\frac{1}{2}},\\ 
\end{aligned}
$$
where $\Lambda$ is a constant. Combining with \eqref{YnY}, we have $\lim_{n\rightarrow \infty}\bar{L}^n_{T}=0.$ 

Similarly, 
$$
\begin{aligned}
	\bar{R}^n_{T}&:= \sup_{(t,x)\in[0,T]\times \mathbb{R}}\left|r^n(t,x)-r(t,x)\right|\\
	&\leq 2C\sup_{t\in[0,T]}\mathbf{E}\left[\left|\widetilde{Y}^n_t-\widetilde{Y}_t\right|\right]\rightarrow 0, \ \ n \rightarrow \infty.
\end{aligned}
$$
 In view of Proposition \ref{continuous}, we have 
$$
\begin{aligned}
 \sup_{t\in[0,T]}\left|K_t^{n}-K'_t\right|
 &\leq \frac{4C}{c}\sup_{t\in[0,T]}\left|s_t^{n}-s_t\right|+\frac{2}{c}\left(\bar{L}^n_{T}\vee\bar{R}^n_{T}\right)\to 0,\ \text{as} \ \  n\to \infty , 
\end{aligned}
$$
which, together with \eqref{YnY}, implies that $ K'\equiv K$. By the uniqueness of BSDE, we obtain that $Y'\equiv Y$. Therefore, the Skorokhod condition holds. The proof is complete.    
\end{proof}

Finally, we consider the general time horizon $[0,T]$ and give the proof of Theorem \ref{globalMFBSDE}.

\begin{proof}[Proof of Theorem \ref{globalMFBSDE}] 
  {\bf Step 1. } We claim that for given $k\in BV[0,T]$, the following MFBSDE with double mean reflections
 \begin{equation}\label{MFBSDEDMRk}
\begin{cases}
Y_t=\xi+\int_t^T f(s,Y_s,\mathbf{P}_{Y_s},Z_s,\mathbf{P}_{Z_s},G_s(k))ds-\int_t^T Z_s dB_s+K_T-K_t, \\
\mathbf{E}[L(t,Y_t)]\leq 0\leq \mathbf{E}[R(t,Y_t)], \\
K_t=K^R_t-K^L_t,\ K^R,K^L\in I[0,T],\\
\int_0^T \mathbf{E}[R(t,Y_t)]dK_t^R=\int_0^T \mathbf{E}[L(t,Y_t)]dK^L_t=0.\\
\end{cases}
\end{equation}
admits a unique solution $(Y,Z,K)\in\mathcal{S}^{\infty}\times BMO \times BV[0,T]$. The proof is similar with the one for Theorem 2.13 in \cite{LP}. For the reader's convenience, we restate the complete proof here.

Indeed, according to Theorem $\ref{thmbddterminal}$, we take an appropriate constant $\bar{h}>0$ depending on $\tilde{A}_0,\hat{H},C,c,M,\lambda$, where $$\tilde{A}_0=(3+\frac{24C}{c})\bar{H}+\frac{2M}{c}+(\frac{2\bar{H}}{\lambda}+\frac{1}{3})e^{9\lambda \bar{H}},$$
$\bar{H}$ is the constant defined in \eqref{barH} and $\hat{H}=\tilde{H}+\lambda\sup_{t\in[0,T]}\left|k_t\right|$. 

Then the doubly mean reflected MFBSDE \eqref{MFBSDEDMRk} admits a unique solution $(Y^1,Z^1,K^1)\in\mathcal{S}^{\infty}[T-\bar{h},T]\times BMO[T-\bar{h},T]\times BV[T-\bar{h},T]$. Moreover, according to Lemma \ref{lem4.6}, we have $\left\|Y^1\right\|_{\mathcal{S}^{\infty}[T-\bar{h},T]}\leq \bar{H}.$ 

Next,  taking $T-\bar{h}$ as the terminal time and applying Theorem $\ref{thmbddterminal}$, the doubly mean reflected MFBSDE \eqref{MFBSDEDMRk} admits a unique solution $(Y^2,Z^2,K^2)\in\mathcal{S}^{\infty}[T-2\bar{h},T-\bar{h}]\times BMO[T-2\bar{h},T-\bar{h}]\times BV[T-2\bar{h},T-\bar{h}]$ to Eq. \eqref{MFBSDEDMRk}. 

Repeating this procedure, for $i=1,...,[\frac{T}{\bar{h}}]$, we set 
\begin{equation}\begin{split}
 &Y_t=\sum_{i=1}^{[\frac{T}{\bar{h}}]} Y_t^i \mathbf{1}_{[T-i\bar{h},T-(i-1)\bar{h})}+Y_T^1 \mathbf{1}_{\left\{T\right\}}, \ Z_t=\sum_{i=1}^{[\frac{T}{\bar{h}}]} Z_t^i \mathbf{1}_{[T-i\bar{h},T-(i-1)\bar{h})}+Z_T^1 \mathbf{1}_{\left\{T\right\}},\\
\end{split}  
\end{equation}
and $K_t=\sum_{i=1}^{[\frac{T}{\bar{h}}]}(\sum_{m=1}^{i-1}K^{m}_{T-(m-1)}\bar{h}+K_t^i)\mathbf{1}_{[T-i\bar{h},T-(i-1)\bar{h})}+K_T^1 \mathbf{1}_{\left\{T\right\}}$. It is easy to check that  $(Y,Z,K)\in\mathcal{S}^{\infty}\times BMO \times BV[0,T]$ is a solution to \eqref{MFBSDEDMRk}. 

According to the uniqueness of local solution on each time interval, the uniqueness of the global solution on $[0,T]$ is obtained.

{\bf Step 2.} For $k^i\in BV[0,T],\ i=1,2$, let $(Y^i,Z^i,K^i)$ be the unique solution to \eqref{MFBSDEDMRk} with data $k^i$. According to Lemma \ref{4.1.1} , there exists a constant $\tilde{h}>0$ depending only on $H,\tilde{H},\lambda, C,c,M$ and $\sup_{t\in[0,T]}\left|k_t^1\right|, \sup_{t\in[0,T]}\left|k_t^2\right|$, such that for any $1 \leq i \leq N$, where $N$ is the smallest integer satisfying $N\geq \frac{T}{\tilde{h}}$,
$$
\begin{aligned}
  &\left\|Y^{1}-Y^{2}\right\|_{\mathcal{S}^{\infty}[(i-1)\tilde{h},i\tilde{h} \wedge T]}+\left\|Z^{1}-Z^{2}\right\|_{BMO[(i-1)\tilde{h},i\tilde{h} \wedge T]}+\sup_{t\in[(i-1)\tilde{h},i\tilde{h} \wedge T]}\left|K_{t}^{1}-K_{t}^{2}\right|\\
  &\leq \frac{1}{2} \left(\left\|Y^{1}-Y^{2}\right\|_{\mathcal{S}^{\infty}[(i-1)\tilde{h},i\tilde{h} \wedge T]}+\left\|Z^{1}-Z^{2}\right\|_{BMO[(i-1)\tilde{h},i\tilde{h} \wedge T]}+\sup_{t\in[(i-1)\tilde{h},i\tilde{h} \wedge T]}\left|k_t^1-k_t^2\right|\right).
\end{aligned}
$$
Therefore, we have 
$$\sup_{t\in[0,T]}\left|K_t^1-K_t^2\right|\leq \max_{1\leq i \leq N}\left(\sup_{t\in[(i-1)\tilde{h},i\tilde{h} \wedge T]}\left|K_t^1-K_t^2\right|\right)\leq \frac{1}{2}\sup_{t\in[0,T]}\left|k_t^1-k_t^2\right|.$$
By the principle of standard contraction mapping, the doubly mean reflected MFBSDE with resistance \eqref{nonlinearyz} admits a unique solution $(Y,Z,K)\in\mathcal{S}^{\infty}\times BMO \times BV[0,T]$. 
\end{proof}


\section{Doubly mean reflected MFBSDEs with density function}
In this section, we consider the well-posedness of doubly mean reflected MFBSDE with resistance given in terms of density function. Before investigating this problem, we first consider the regularity of the last component of the solution  to MFBSDEs with double mean reflections. For this purpose, we introduce the following assumption for the loss functions.

\begin{assumption}\label{assLRbounded}
    The loss functions $L,R \in C^{1,2}_b([0,T]\times\mathbb{R})$, where $C^{1,2}_b([0,T]\times\mathbb{R})$ is the collection  of functions that are continuously differentiable in their first variable and twice continuously differentiable in their second variable, and both
derivatives are uniformly bounded.
\end{assumption}
\begin{proposition}\label{Klip}
    Suppose that Assumptions \ref{asslip2}, \ref{assLRbounded} hold.  Let $(Y,Z,K)$ denote the unique deterministic solution to MFBSDE \eqref{MFBSDEDMR} with double mean reflections. Moreover, suppose that 
    \begin{align}\label{bdd of Z}
        \sup_{t\in[0,T]}\E[|Z_t|^2]<\infty.
    \end{align}
Then the process $K$ is Lipschitz continuous. Furthermore, $K$ has the density function.
\end{proposition}
\begin{proof}
    According to Proposition 4.2 in \cite{L}, solution $K$ is the second component of solution of backward Skorokhod problem $\mathbb{BSP}_{l}^{r}(s,a)$. Here, for each $t\in[0,T]$,
\begin{equation}\begin{split}
&s_t=\mathbf{E}\left[\int_{0}^{t}f(s,Y_s,\mathbf{E}[Y_s],Z_s,\mathbf{E}[Z_s])ds\right],\ a=\mathbf{E}[\xi],\\
&l(t,x):=\mathbf{E}\left[L(t,y_t-\E[y_t]+x)\right], \ r(t,x):=\mathbf{E}\left[R(t,y_t-\E[y_t]+x)\right], 
\end{split}
 \end{equation}
where $(y,z)$ is the solution to  BSDE with terminal $\xi$ and $\left\{f(t,Y_t,\mathbf{E}[Y_t],Z_t,\mathbf{E}[Z_t])\right\}_{t\in[0,T]}$. 
For any $t\in[0,T]$, $K_t=\bar{K}_T-\bar{K}_{T-t}$, where $\bar{K}$ is the second component of the solution to the Skorokhod problem $\mathbb{SP}_{\bar{l}}^{\bar{r}}(\bar{s})$. Here, for any $t\in[0,T]$, 
\begin{equation}
\bar{s}_{t}=a+s_{T}-s_{T-t}=\mathbf{E}[y_{T-t}],\ \bar{l}(t,x)=l (T-t,x), \ \bar{r}(t,x)=r(T-t,x).
\end{equation}
Let {$\bar{\phi}_{t}, \bar{\psi}_{t}$}  be the unique solutions to the following equations, respectively
$$\bar{l}(t,\bar{s}_t+x)=0, \ \bar{r}(t,\bar{s}_t+x)=0.$$
Applying Remark 2.7 in \cite{Li}, we have 
\begin{equation}
\begin{split}
   \left|K_t-K_s\right|&\leq \sup_{u_1,u_2\in[s,t]}\left|K_{u_1}-K_{u_2}\right|\\
   &=\sup_{u_1,u_2\in[s,t]}\left|\bar{K}_T-\bar{K}_{T-u_1}-(\bar{K}_T-\bar{K}_{T-u_2})\right|\\
   &\leq \sup_{u_1,u_2\in[s,t]}\left|\bar{K}_{T-u_1}-\bar{K}_{T-u_2}\right|=\sup_{t_1,t_2\in[T-t,T-s]}\left|\bar{K}_{t_1}-\bar{K}_{t_2}\right|\\
   &\leq \sup_{p_1,p_2\in[T-t,T-s]}\left|\bar{\phi}_{p_1}-\bar{\phi}_{p_2}\right|+\sup_{p_1,p_2\in[T-t,T-s]}\left|\bar{\psi}_{p_1}-\bar{\psi}_{p_2}\right|\\
   &{= \sup_{p_1,p_2\in[s,t]}\left|{\phi}_{p_1}-{\phi}_{p_2}\right|+\sup_{p_1,p_2\in[s,t]}\left|{\psi}_{p_1}-{\psi}_{p_2}\right|},
\end{split}    
\end{equation}
where $\bar{\phi}_t=\phi_{T-t}$ and ${\psi}_t=\bar{\psi}_{T-t}$. It remains to prove that $\phi_t,\psi_t$ are Lipschitz continuous. We only prove this fact for $\phi$. First, we have
\begin{equation}
\begin{split}
    \left|\phi_{t}-\phi_{s}\right|&=\left|\bar{\phi}_{T-t}-\bar{\phi}_{T-s}\right|\\
    &\leq  \frac{1}{c}\left|\bar{l}(T-t,\bar{s}_{T-t}+\bar{\phi}_{T-t})-\bar{l}(T-t,\bar{s}_{T-t}+\bar{\phi}_{T-s})\right|\\
    &=\frac{1}{c}\left|\bar{l}(T-t,\bar{s}_{T-t}+\bar{\phi}_{T-s})\right|\\
    &=\frac{1}{c}\left|\mathbf{E}\left[L(t,y_{t}+\phi_{s})\right]\right|.
\end{split}    
\end{equation}
Applying It\^o's  formula to $L(\cdot,y_{\cdot}+\phi_{s})$ on the interval $[s,t]$ and then taking expectations on both sides of the equation, combining with Assumptions \ref{asslip2},\ \ref{assLRbounded} and \eqref{bdd of Z}, we have
\begin{equation}
\begin{split}
  \left|\mathbf{E}\left[L(t,y_{t}+\phi_{s})\right]\right|&\leq  \mathbf{E}\left[\int_{s}^{t}\left|\frac{\partial L}{\partial r}(r,y_{r}+\phi_{s})\right|dr\right]+\frac{1}{2}\mathbf{E}\left[\int_{s}^{t}\left|\frac{\partial^2 L}{\partial x^2}(r,y_{r}+\phi_{s})Z_r^2\right|dr\right]\\ 
  &\quad+\mathbf{E}\left[\int_{s}^{t}\left|\frac{\partial L}{\partial x}(r,y_{r}+\phi_{s})\left(|Y_r|+\mathbf{E}[|Y_r|]+|Z_r|+\mathbf{E}[|Z_r|]\right)\right|dr\right]\\
  &\quad+\mathbf{E}\left[\int_{s}^{t}\left|\frac{\partial L}{\partial x}(r,y_{r}+\phi_{s})f(r,0,0,0,0)\right|dr\right]\\
  &\leq \Lambda |t-s|.
\end{split}
\end{equation}
where $\Lambda$ is an appropriate constant.
The proof is complete. 
\end{proof}
\begin{lemma}\label{estimateDYDZ}
 Consider the following MFBSDE:
\begin{equation}\label{MFBSDEm}
  Y_t=\xi+\int_t^T f(s,Y_s,\mathbf{E}[Y_s],Z_s,\mathbf{E}[Z_s])ds-\int_t^T Z_s dB_s.  
\end{equation}
Let Assumption \ref{asslip2} hold. Suppose that $f$ is continuously differentiable in $y$ and $z$ with uniformly bounded derivative. Moreover, assume that $\xi$ and $f(\cdot, y,y',z,z')$ are Malliavin differentiable for each $y,y',z,z'$ with 
    \begin{itemize}
        \item [(i)] $\sup_{\theta \in[0,T]}\mathbf{E}\left[|D_{\theta}\xi|^p\right]<\infty$;
        \item [(ii)] $\sup_{\theta \in[0,T]}\mathbf{E}\left[\left(\int_{0}^{T}|D_{\theta}f(t,Y_t,\mathbf{E}[Y_t],Z_t,\mathbf{E}[Z_t])|dt\right)^{p}\right]<\infty$.
    \end{itemize}  
    Then, we have $\mathbf{E}\left[\sup_{t\in[0,T]}\left|Z_t\right|^p\right]<\infty$. 
\end{lemma}

\begin{proof}
      First, similar to Proposition 5.3 in \cite{EPQ}, it is easily check that the Malliavin derivatives $\left\{(D_{\theta}Y_t,D_{\theta}Z_t)\right\}_{0\leq \theta,t\leq T}$ satisfies the following linear BSDE:  
\begin{equation}\label{eqDYDZ}
\begin{split}
D_\theta Y_t &= D_\theta \xi + \int_t^T \Big[ \partial_y f(r,Y_r,\mathbf{E}[Y_r],Z_r,\mathbf{E}[Z_r])D_\theta Y_r+\partial_z f(r,Y_r,\mathbf{E}[Y_r],Z_r,\mathbf{E}[Z_r]) D_\theta Z_r \\
&\quad+ D_\theta f(r,Y_r,\mathbf{E}[Y_r],Z_r,\mathbf{E}[Z_r]) \Big] dr-\int_t^T D_\theta Z_r dB_r, \qquad 0 \leq \theta \leq t \leq T;
\end{split}
\end{equation}
\begin{equation}\label{eqDYDZ0}
D_\theta Y_t = 0, \qquad D_\theta Z_t = 0, \qquad 0 \leq t < \theta \leq T.
\end{equation}
According to Lemma \ref{proMFBSDE}, we have 
\begin{equation}\label{eqDYDZ1}
\begin{split}
&\mathbf{E}\left[\sup_{t\in[0,T]}\left|D_{\theta}Y_t\right|^p\right]+\mathbf{E}\left[\left(\int_{0}^{T}\left|D_{\theta}Z_t\right|^2dt\right)^{\frac{p}{2}}\right]\\
&\leq C_p\left(\mathbf{E}\left[\left|D_{\theta}\xi\right|^p\right]+\mathbf{E}\left[\left(\int_{0}^{T}|D_{\theta}f(r,Y_r,\mathbf{E}[Y_r],Z_r,\mathbf{E}[Z_r])|dt\right)^{p}\right]\right)<\infty, \  \forall \theta\in[0,T].\\
\end{split}
\end{equation}
For $t\leq s$, we have
$$Y_s=Y_t-\int_{t}^{s}f(r,Y_r,\mathbf{E}[Y_r],Z_r,\mathbf{E}[Z_r])dr+\int_{t}^{s}Z_rdB_r.$$
Therefore, for $t< \theta \leq s$,
\begin{equation}
\begin{split}
D_\theta Y_s &= Z_{\theta}- \int_{\theta}^{s} \Big[ \partial_y f(r,Y_r,\mathbf{E}[Y_r],Z_r,\mathbf{E}[Z_r])D_\theta Y_r+\partial_z f(r,Y_r,\mathbf{E}[Y_r],Z_r,\mathbf{E}[Z_r]) D_\theta Z_r \\
&\quad+ D_\theta f(r,Y_r,\mathbf{E}[Y_r],Z_r,\mathbf{E}[Z_r]) \Big] dr+\int_{\theta}^{s} D_\theta Z_r dB_r;
\end{split}
\end{equation}
Then, taking $\theta=s$, we have $D_sY_s=Z_s, \ a.s.$  Therefore,  $\mathbf{E}\left[\sup_{t\in[0,T]}\left|Z_t\right|^p\right]<\infty$.
\end{proof}

\begin{remark}
 For example, let us consider MFBSDE with the terminal condition $\xi=g(X_T)$ and the   generator $f(t,y,y',z,z')=F(X_t,y,y',z,z')$, where $X_t$ satisfies the following SDE:
    \begin{equation*}
        X_t=x_0+\int_{0}^{t}b(X_s)ds+\int_{0}^{t}\sigma(X_s)dB_s,
    \end{equation*}
 and the functions $b,\sigma,F,g$ are deterministic and continuously differentiable such that $\partial_xb$, $\partial_{x}\sigma$,  $\partial_{x}g$, $\partial_{x}F$,  $\partial_{y}F$, $\partial_{z}F$ are bounded. Then, the assumptions in Lemma \ref{estimateDYDZ} are satisfied. Therefore, we can derive 
  $\mathbf{E}\left[\sup_{t\in[0,T]}\left|Z_t\right|^p\right]<\infty.$ 
  
In particular, adding a deterministic function $A$ into Eq. \eqref{MFBSDEm}, i.e., $(Y,Z)$ satisfies the following equation
\begin{align*}
    Y_t=\xi+\int_t^T f(s,Y_s,\mathbf{E}[Y_s],Z_s,\mathbf{E}[Z_s])ds-\int_t^T Z_s dB_s+A_T-A_t,
\end{align*}
the conclusion of Lemma \ref{estimateDYDZ} remains valid. Therefore, consider the MFBSDE with double mean reflections:
\begin{equation}\label{EMFBSDE}
\begin{cases}
Y_t=\xi+\int_t^T f(s,Y_s,\mathbf{E}[Y_s],Z_s,\mathbf{E}[Z_s])ds-\int_t^T Z_s dB_s+K_T-K_t, \\
\mathbf{E}[L(t,Y_t)]\leq 0\leq \mathbf{E}[R(t,Y_t)],\\ 
{K_t=K^+_t-K^-_t, \ t\in[0,T]}\\
\int_0^T \mathbf{E}[R(t,Y_t)]dK^+_t=\int_0^T \mathbf{E}[L(t,Y_t)]dK^-_t=0. 
\end{cases}
\end{equation}
Under the same assumptions for $\xi$ and $f$ as in Lemma \ref{estimateDYDZ}, we still have $\mathbf{E}\left[\sup_{t\in[0,T]}\left|Z_t\right|^p\right]<\infty$.. 
 \end{remark}
   
Suppose that the compensating term $K$ of the doubly mean reflected MFBSDE with nonlinear resistance is absolutely continuous w.r.t. Lebesgue measure, i.e., there exists a density process $k$ such that $K_t=\int_0^t k_sds$. In this section, we consider doubly mean reflected MFBSDE \eqref{nonlinearyz} when the nonlinear resistance is given as $G_s(K)=k_s$ and we focus on the case $p=2$. In this case, doubly mean reflected MFBSDE \eqref{nonlinearyz} can be rewritten as follows:
\begin{equation}\label{mfbsdefk}
\begin{cases}
Y_t=\xi+\int_t^T \left(f(s,Y_s,\mathbf{E}[Y_s],Z_s,\mathbf{E}[Z_s],k_s)+k_s\right)ds-\int_t^T Z_s dB_s, \\
\mathbf{E}[L(t,Y_t)]\leq 0\leq \mathbf{E}[R(t,Y_t)],\\ 
{k_t=k^+_t-k^-_t, \ t\in[0,T] \textrm{ and }k^+,k^-\geq 0,}\\
\int_0^T \mathbf{E}[R(t,Y_t)]k_t^+dt=\int_0^T \mathbf{E}[L(t,Y_t)]k^-_tdt=0.
\end{cases}
\end{equation}
Unfortunately, the nonlinear resistance function $G$ does not satisfy Assumption \ref{assgk}. To obtain well-posedness of Eq.\eqref{mfbsdefk}, we will construct the solution via the following auxiliary doubly mean reflected MFBSDE without resistance
\begin{equation}\label{aumdrmfbsdewr}
\begin{cases}
\widetilde{Y}_t=\xi+\int_t^T f(s,\widetilde{Y}_s,\mathbf{E}[\widetilde{Y}_s],\widetilde{Z}_s,\mathbf{E}[\widetilde{Z}_s],0)ds-\int_t^T\widetilde{Z}_s dB_s+\widetilde{K}_T-\widetilde{K}_t, \\
\mathbf{E}[L(t,\widetilde{Y}_t)]\leq 0\leq \mathbf{E}[R(t,\widetilde{Y}_t)],\\ 
\widetilde{K}_t=\widetilde{K}^{+}_t-\widetilde{K}^{-}_t,\ \widetilde{K}^+,\widetilde{K}^{-} \in I[0,T],\\
\int_0^T \mathbf{E}[R(t,\widetilde{Y}_t)]d\widetilde{K}_t^+=\int_0^T \mathbf{E}[L(t,\widetilde{Y}_t)]d\widetilde{K}^-_t=0.
\end{cases}
\end{equation}

According to Theorem 3.5 in \cite{LS}, the doubly mean reflected MFBSDE \eqref{aumdrmfbsdewr} admits  a unique solution $(\widetilde{Y},\widetilde{Z},\widetilde{K})\in \mathcal{S}^2 \times \mathcal{H}^2 \times BV[0,T]$. Suppose additionally that the loss functions $L,R\in C^{1,2}_b([0,T]\times\mathbb{R})$ and 
\begin{align}\label{bdd of tilde Z}
    \sup_{t\in[0,T]}\E[|\widetilde{Z}_t|^2]<\infty.
\end{align}
According to Proposition \ref{Klip},  there exist three measurable functions  $\widetilde{k},\widetilde{k}^+,\widetilde{k}^-$ with $\widetilde{k}^+,\widetilde{k}^-\geq0$, such that
\begin{align*}
\widetilde{K}_t=\int_{0}^{t}\widetilde{k}_s ds, \ \widetilde{K}^+_t=\int_{0}^{t}\widetilde{k}^+_s ds, \ \widetilde{K}^-_t=\int_{0}^{t}\widetilde{k}^-_s ds.
\end{align*}
Thus, we rewrite Eq. \eqref{aumdrmfbsdewr} as follows:
\begin{equation}
\begin{cases}
\widetilde{Y}_t=\xi+\int_t^T \left(f(s,\widetilde{Y}_s,\mathbf{E}[\widetilde{Y}_s],\widetilde{Z}_s,\mathbf{E}[\widetilde{Z}_s],0)+\widetilde{k}_s\right)ds-\int_t^T\widetilde{Z}_s dB_s, \\
\mathbf{E}[L(t,\widetilde{Y}_t)]\leq 0\leq \mathbf{E}[R(t,\widetilde{Y}_t)],\\ 
{\widetilde{k}_t=\widetilde{k}^+_t-\widetilde{k}^-_t, \ \widetilde{k}^+,\widetilde{k}^-\geq 0,} \\
\int_0^T \mathbf{E}[R(t,\widetilde{Y}_t)]\widetilde{k}^+_tdt=\int_0^T \mathbf{E}[L(t,\widetilde{Y}_t)]\widetilde{k}^-_tdt=0.
\end{cases}
\end{equation}

\begin{theorem}
Let Assumptions \ref{asslip2}, \ref{assLRbounded} and \eqref{bdd of tilde Z} hold. Moreover, suppose that for any $s\in[0,T]$, each of the following two equations admits a unique deterministic nonnegative solution
\begin{equation}\label{k+k-}
    \begin{split}
        &f(s,\widetilde{Y}_s,\mathbf{E}[\widetilde{Y}_s],\widetilde{Z}_s,\mathbf{E}[\widetilde{Z}_s],k)-f(s,\widetilde{Y}_s,\mathbf{E}[\widetilde{Y}_s],\widetilde{Z}_s,\mathbf{E}[\widetilde{Z}_s],0)+k=\widetilde{k}_s^+,\\
        &f(s,\widetilde{Y}_s,\mathbf{E}[\widetilde{Y}_s],\widetilde{Z}_s,\mathbf{E}[\widetilde{Z}_s],-k)-f(s,\widetilde{Y}_s,\mathbf{E}[\widetilde{Y}_s],\widetilde{Z}_s,\mathbf{E}[\widetilde{Z}_s],0)-k=-\widetilde{k}_s^-.
    \end{split}
\end{equation}
Then the mean doubly reflected MFBSDE \eqref{mfbsdefk} with nonlinear resistance has at least one solution. 
\end{theorem}
\begin{proof}
Let $R_t=\{\E[R(t,Y_t)]=0\}$, $L_t=\{\E[L(t,Y_t)]=0\}$ and $A_t=\{\E[L(t,Y_t)]<0<\E[R(t,Y_t)]\}$. It is easy to check that for any $t$, $\mathbf{1}_{R_t}+\mathbf{1}_{L_t}+\mathbf{1}_{A_t}=1$ and 
\begin{align*}
    k_t=\begin{cases}
        k^+_t, & t\in R_t,\\
        0, &t\in A_t,\\
        -k^-_t, &t\in L_t.
    \end{cases}
\end{align*}
Therefore, we may rewrite the first equation to mean doubly reflected MFBSDE \eqref{mfbsdefk} with nonlinear resistance as follows:
    \begin{equation}
        \begin{split}
            Y_t&=\xi+\int_{t}^{T}\left(f(s,Y_s,\mathbf{E}[Y_s],Z_s,\mathbf{E}[Z_s],k_s)+k_s\right) ds-\int_{t}^{T}Z_sdB_s\\
           &=\xi+\int_{t}^{T}f(s,Y_s,\mathbf{E}[Y_s],Z_s,\mathbf{E}[Z_s],0) ds-\int_{t}^{T}Z_sdB_s\\
           &\quad +\int_{t}^{T}\mathbf{1}_{R_t}\left(f(s,Y_s,\mathbf{E}[Y_s],Z_s,\mathbf{E}[Z_s],k^+_s)-f(s,Y_s,\mathbf{E}[Y_s],Z_s,\mathbf{E}[Z_s],0)+k_s^+\right)ds\\
            &\quad +\int_{t}^{T}\mathbf{1}_{L_t}\left(f(s,Y_s,\mathbf{E}[Y_s],Z_s,\mathbf{E}[Z_s],{-k^-_s})-f(s,Y_s,\mathbf{E}[Y_s],Z_s,\mathbf{E}[Z_s],0)-k_s^-\right)ds.
        \end{split}
    \end{equation} 
For any $s\in[0,T]$, let $Y_s=\widetilde{Y}_s, Z_s=\widetilde{Z}_s$ and $k_s=k^+_s-k^-_s$, where $k^+_s,k^-_s$ are the solutions to equations in \eqref{k+k-}, respectively. We claim that $(Y,Z,k)$ is a solution to \eqref{mfbsdefk}. It remains to prove that $k^+,k^-$ satisfy the flat-off condition. In fact, for any $s\in [0,T]$, it is easy to check that $k^+_s>0$ (resp., $k^-_s>0$) if and only if $\widetilde{k}^+_s>0$ (resp., $\widetilde{k}^-_s>0$). Therefore, the flat-off conditions hold. 
\end{proof}
\begin{remark}
  Suppose that $f$ satisfying Assumption \ref{asslip2} has the following representation 
  \begin{align}\label{repre of f}f(t,\omega,y,y',z,z',k)=a(t,\omega,y,y',z,z')+b(t,y',z',k).
  \end{align}
  For each $t\in[0,T]$, set $\bar{b}_t(k)=b(t,\mathbf{E}[\widetilde{Y}_t],\mathbf{E}[\widetilde{Z}_t],k)$. For simplicity, we omit $t$ in the subscript. Then, equations in  \eqref{k+k-} are solvable if one of the following  conditions holds:
\begin{itemize}
    \item [(i)] $\bar{b}$ is increasing in $k$;
    \item[(ii)] there exists a constant $C_k<1$, such that $\left|\bar{b}(k^1)-\bar{b}(k^2)\right|\leq C_k\left|k^1-k^2\right|$;
\end{itemize} 

For $k\geq 0$, we set $ \widetilde{b}(k)=\bar{b}(k)-\bar{b}(0)+k$ and $\widehat{b}(k)=\bar{b}(-k)-\bar{b}(0)-k$. Given a constant $c>0$, it remains to show that each of the following equations $\widetilde{b}(k)=c$ and $\widehat{b}(k)=-c$ admits a strictly positive solution. We only prove this fact for $\widetilde{b}$. It suffices to show that the mapping $\widetilde{b}:[0,\infty)\rightarrow[0,\infty)$ is a one-to-one correspondence. First, $\widetilde{b}(0)=0$ clearly holds for all cases. 

  Suppose that (i) holds. Then, it is easy to check that $\widetilde{b}$ is strictly increasing and $\lim_{k\rightarrow\infty}\widetilde{b}(k)=\infty$. 

Suppose that (ii) holds. For $0\leq k < k'$, we have 
\begin{equation*}
    \begin{split}
       \widetilde{b}(k')-\widetilde{b}(k)&= \bar{b}(k')-\bar{b}(k)+k'-k\\
   &\geq -C_k\left|k'-k\right|+k'-k\\
   &=(1-C_k)(k'-k)> 0. 
    \end{split}
\end{equation*}
Then $\widetilde{b}$ is strictly increasing in $k$ and $\lim_{k\rightarrow\infty}\widetilde{b}(k)=\infty$.


\end{remark}

For uniqueness, we only deal with the case where the loss functions are linear. More precisely, consider the following doubly mean reflected MFBSDE with resistance:
\begin{equation}\label{drmfbsde}
\begin{cases}
Y_t=\xi+\int_t^T \left(f(s,Y_s,\mathbf{E}[Y_s],Z_s,\mathbf{E}[Z_s],k_s)+k_s \right)ds-\int_t^T Z_s dB_s, \\
l_t\leq \mathbf{E}[Y_t]\leq r_t, \ t\in[0,T],\\ 
{k_t=k^+_t-k^-_t, \ t\in[0,T] \textrm{ and }k^+,k^-\geq 0,}\\
\int_0^T (\mathbf{E}[Y_t]-l_t)k_t^{+}dt=\int_0^T (r_t-\mathbf{E}[Y_t])k_t^{-}dt=0.
\end{cases}
\end{equation}

\begin{proposition}\label{prounik}
Let Assumption \ref{asslip2} hold. Moreover, suppose that $f$ has representation \eqref{repre of f} and it satisfies $\frac{df}{dk}\geq-1$. 
   Then the doubly mean reflected MFBSDE \eqref{drmfbsde} with resistance admits at most one unique solution.
\end{proposition}

\begin{proof}
    Suppose that $(Y^1,Z^1,k^1)$ and $(Y^2,Z^2,k^2)$ are two solutions to doubly mean reflected MFBSDE  \eqref{drmfbsde} with resistance. In order to simplify notations, denote
    $$\Delta K:=K^1-K^2,  \  f^{i,j}_s:=f(s,Y_s^i,\mathbf{E}[Y_s^i],Z_s^i,\mathbf{E}[Z_s^i],k_s^j),\text{ where }\ K= Y,\ Z,\ k,\ f;\ i,j=1,2$$
    and 
    \begin{equation*}
     \begin{split}
         \alpha^{1,2}_s:=\frac{f^{2,1}_s-f^{2,2}_s}{\Delta k_s} I_{\{\Delta k_s\neq 0\}}. 
         \end{split}   
    \end{equation*}
    
    Applying It\^o formula to $\left|\Delta Y \right|^2$ and taking expectation, we obtain 
    \begin{equation}
    \begin{split}
      & \mathbf{E}\left[\left|\Delta Y_t\right|^2\right]+\mathbf{E}\left[\int_{t}^{T}\left|\Delta Z_s\right|^2ds\right]\\
      =&2\mathbf{E}\left[\int_{t}^{T}\Delta Y_s \Delta f_s ds\right] +2\mathbf{E}\left[\int_{t}^{T}\Delta Y_s \Delta k_s ds\right]\\
      =&2\mathbf{E}\left[\int_{t}^{T}\Delta Y_s(f_s^{1,2}-f_{s}^{2,2})ds\right]+2\mathbf{E}\left[\int_{t}^{T}(1+\alpha^{1,2}_s)\Delta Y_s \Delta k_s ds\right]\\
      \leq & 2\lambda\mathbf{E}\left[\int_{t}^{T}\Delta Y_s(\left|\Delta Y_s \right|+\mathbf{E}[\left|\Delta Y_s\right|])ds\right]+2\lambda\mathbf{E}\left[\int_{t}^{T}\Delta Y_s(\left|\Delta Z_s\right|+\mathbf{E}[\left|\Delta Z_s\right|])ds\right]\\
      &\quad +\left[2\mathbf{E}\int_{t}^{T}(1+\alpha^{1,2}_s)\Delta Y_s \Delta k_s ds\right].\\
    \end{split}
    \end{equation}
    By the assumption $f$, $\alpha^{1,2}$ is a deterministic function and we have $1+\alpha^{1,2}_s\geq 0$. By the Skorokhod conditions, we have $k^{i,+}_s(\mathbf{E}[Y^i_s]-l_s)=0, \ a.e.$ and $k^{i,-}_s(r_s-\mathbf{E}[Y^i_s])=0, \ a.e.$, $i=1,2$.   Therefore,
  \begin{equation*}
      \begin{split}
          &\E\left[\int_t^T(1+\alpha_s^{1,2})\Delta Y_s \Delta k_s ds\right] =\int_t^T(1+\alpha_s^{1,2})\E\left[\Delta Y_s\right] \Delta k_sds\\
          =&\int_{t}^{T}(1+\alpha_s^{1,2})(k_s^{1,+}-k_s^{2,+})(\mathbf{E}[Y^1_s]-l_s+l_s-\mathbf{E}[Y^2_s])ds \\
           &+\int_{t}^{T}(1+\alpha_s^{1,2})(k_s^{2,-}-k_s^{1,-})(\mathbf{E}[Y^2_s]-r_s+r_s-\mathbf{E}[Y^1_s])ds \\
           =&-\int_{t}^{T}(1+\alpha_s^{1,2})(\mathbf{E}[Y^2_t]-l_s)k_s^{1,+}ds-\int_{t}^{T}(1+\alpha_s^{1,2})(\mathbf{E}[Y^2_t]-l_s)k_s^{2,+}ds\\
           &-\int_{t}^{T}(1+\alpha_s^{1,2})(r_s-\mathbf{E}[Y^2_t])k_s^{1,-}ds-\int_{t}^{T}(1+\alpha_s^{1,2})(r_s-\mathbf{E}[Y^2_t])k_s^{2,-}ds\leq 0.
      \end{split}
  \end{equation*}
   All the above analysis indicates that 
 \begin{equation}
    \begin{split}
       \mathbf{E}\left[\left|\Delta Y_t\right|^2\right]+\mathbf{E}\left[\int_{t}^{T}\left|\Delta Z_s\right|^2ds\right]
      \leq  (4\lambda+4\lambda^2)\mathbf{E}\left[\int_{t}^{T}\left|\Delta Y_s \right|^2 ds\right]+\mathbf{E}\left[\int_{t}^{T}\left|\Delta Z_s\right|^2ds\right].
    \end{split}
    \end{equation}
It follows that 
\begin{equation*}  
\mathbf{E}\left[|\Delta Y_t|^2\right]\leq (4\lambda+4\lambda^2)\int_{t}^{T}\mathbf{E}\left[|\Delta Y_t|^2\right]ds.
\end{equation*}
By Gronwall inequality, $Y^1_t=Y^2_t$, on $[0,T]$. Applying It\^{o}'s formula to $(Y^1_t-Y^2_t)^2$, we obtain that $Z^1_t=Z^2_t$ on $t\in[0,T]$ and finally, we have $k^1_t=k^2_t$ on $[0,T]$.
\end{proof}

\section*{Acknowledgement}
    This work was supported  by the National Natural Science Foundation of China (No. 12301178), the Natural Science Foundation of Shandong Province for Excellent Young Scientists Fund Program (Overseas) (No. 2023HWYQ-049),  the Natural Science Foundation of Shandong Province (No. ZR2023ZD35),  the Fundamental Research Funds for the Central Universities and  the Qilu Young Scholars Program of Shandong University. 

\end{document}